\RequirePackage{fix-cm}
\documentclass{svjour3}                     
\smartqed  
\usepackage[margin=1in]{geometry}
\usepackage{graphicx}
\usepackage{amsmath} 
\usepackage{mathtools}
\usepackage{authblk} 
\usepackage{amssymb}
\usepackage{amsfonts}
\usepackage{verbatim}
\usepackage{mathrsfs}
\usepackage{subfigure}
\usepackage{mathtools}
\usepackage{nicefrac}
\usepackage{xcolor}
\usepackage{url}
\usepackage{todonotes}
\usepackage{multirow}
\usepackage{chngcntr}
\usepackage{apptools}
\AtAppendix{\counterwithin{lemma}{section}}
\newcommand{\bydef}{\,\stackrel{\mbox{\tiny\textnormal{\raisebox{0ex}[0ex][0ex]{def}}}}{=}\,}

\newcommand{\C}{\mathbb C}

\begin{document}

\title{From the Lagrange polygon to the figure eight I
\thanks{
RC was partially supported by UNAM-PAPIIT project IN101020.
CGA was partially supported by UNAM-PAPIIT grant IA100121.
JPL was partially supported by an NSERC Discovery Grant.
JDMJ was partially supported by NSF grant 
DMS-1813501.
}
}
%

\subtitle{Numerical evidence extending a conjecture of Marchal}

\author{Renato Calleja \and
Carlos Garc\'{\i}a-Azpeitia \and
Jean-Philippe Lessard \and
J.D. Mireles James 
}


\institute{R. Calleja
               \at Universidad Nacional Aut\'onoma de M\'exico, 
               IIMAS \\
               \email{calleja@mym.iimas.unam.mx}
\and
              C. Garc\'{\i}a-Azpeitia
              \at Universidad Nacional Aut\'onoma de M\'exico, 
               IIMAS \\
               \email{cgazpe@ciencias.unam.mx}
\and
              J.P. Lessard
              \at McGill University, 
              Department of Mathematics and Statistics \\
              \email{jp.lessard@mcgill.ca}
\and
              J.D. Mireles James \at
              Florida Atlantic University, 
              Department of Mathematical Sciences \\
              \email{jmirelesjames@fau.edu}
}

\date{}

\maketitle

\begin{abstract}
The present work studies the continuation class of the regular 
$n$-gon solution of  the $n$-body problem. For odd numbers of bodies between 
$n = 3$ and $n = 15$ we apply one parameter 
numerical continuation algorithms to the energy/frequency variable, 
and find that  the figure eight choreography 
can be reached starting from the regular $n$-gon.  The continuation leaves the 
plane of the $n$-gon, and passes through
families of spatial choreographies with the topology of torus knots.
Numerical continuation out of the $n$-gon solution is complicated by the fact 
that the kernel of the linearization there is high dimensional. 
Our work exploits a symmetrized version of the problem 
which admits dense sets of choreography solutions, and which can be written 
as  a delay differential equation in terms of one of the bodies.  
This symmetrized setup simplifies the problem in several ways. 
On one hand, the direction of the kernel is determined automatically 
by the symmetry.  On the other hand, the set of possible bifurcations
is reduced and the $n$-gon 
continues to the eight after a single symmetry breaking bifurcation.
Based on the calculations presented here 
we conjecture that the $n$-gon and the eight are in the same continuation 
class for all odd numbers of bodies.
\keywords{$n$-body problem
\and choreographies
\and numerical continuation
\and delay differential equations}
\PACS{45.50.Jf	 \and 45.50.Pk \and 45.10.-b \and 02.60.Lj \and 05.45.Ac}
\subclass{70K44 \and 34C45 \and 70F15}
\end{abstract}



%
%
\section{Introduction} 
\label{sec:intro}
The qualitative theory of nonlinear dynamics has deep roots 
in the pioneering work of
Poincar\'{e} \cite{MR926906,MR926907,MR926908}, where
invariant sets  -- and periodic orbits in particular -- play a 
central organizational role.  Inspired by the work of Poincar\'{e},
 a number of late Nineteenth and early Twentieth Century 
astronomers like Darwin, Moulton, and Str\"{o}mgren conducted thorough numerical 
studies of periodic motions in gravitational $n$--body problems long before the 
advent of digital computing \cite{MR1554890,MR0094486,Stromgren}. Over the last 
century scientific interest in $n$-body dynamics  has only increased, driving developments in
diverse fields from computational mathematics to algebraic topology. By now
the literature is rich enough to discourage even a terse survey. We refer
to the Lecture notes of Chenciner \cite{MR3329413}, as well as the books of
Moser, Meyer and Hall, and Szebehely 
\cite{MR1829194,MR2468466,theoryOfOrbits} where the interested reader will
find both modern overviews of the theory and thorough reviews of the literature.

The present work is concerned with a special class of periodic
orbits known as \textit{choreographies}, where $n$ gravitating bodies 
follow one another around the same 
closed curve in $\mathbb{R}^3$.  
The most basic example of a choreography comes from the  
classical equilateral triangle 
configuration of Lagrange, where three massive bodies
are located at the vertices of a rigid equilateral
triangle revolving around the center of mass.  
If the bodies all have the same mass
then each goes around the same circle with constant 
angular velocity.  This is an example of a circular choreography.  See the left frame 
of Figure \ref{fig:Polygons}.

Lagrange published this special solution of the three body problem in 1772 
\cite{lagrangeTriangle}.  The result was generalized by Hoppe in 1879
\cite{nGonPaper}, giving the existence of a circular choreography
for any number of bodies.  In this case the bodies are arranged at the vertices 
of a rotating regular $n$-gon, and the choreography is the inscribing 
circle.  In 1985 Perko and Walter \cite{PerkoW85} showed that when $n \geq 4$, 
in contrast to the 3 body case, the $n$-gon solution exists if and only if the masses 
of the $n$ bodies are equal.  The right frame of 
Figure  \ref{fig:Polygons} illustrates a circular $n$-gon choreography solution 
for the case of $n = 15$ bodies.

The first non-circular choreography solution was discovered numerically by
C.~Moore in the early 1990's \cite{Mo93}. This solution consists of three
bodies of equal mass following one another around the now famous figure-eight orbit.
See the left frame of Figure~\ref{fig:Eights}. Chenciner and Montgomery, in \cite{ChMo00}
gave a rigorous mathematical proof of the existence of the figure-eight
choreography by minimizing the Newtonian action functional over paths connecting
collinear and isosceles configurations of the three bodies. Many additional
numerical results for the eight are described by Sim\'{o} in \cite{MR1884902},
who also coined the term {\em choreography}.
Several animations of $n$-body choreographies are found at the webpage
\cite{renatoAnimations}.

\begin{figure}[t!]
\centering
\vspace{\baselineskip}
\includegraphics[width=.7 \textwidth]{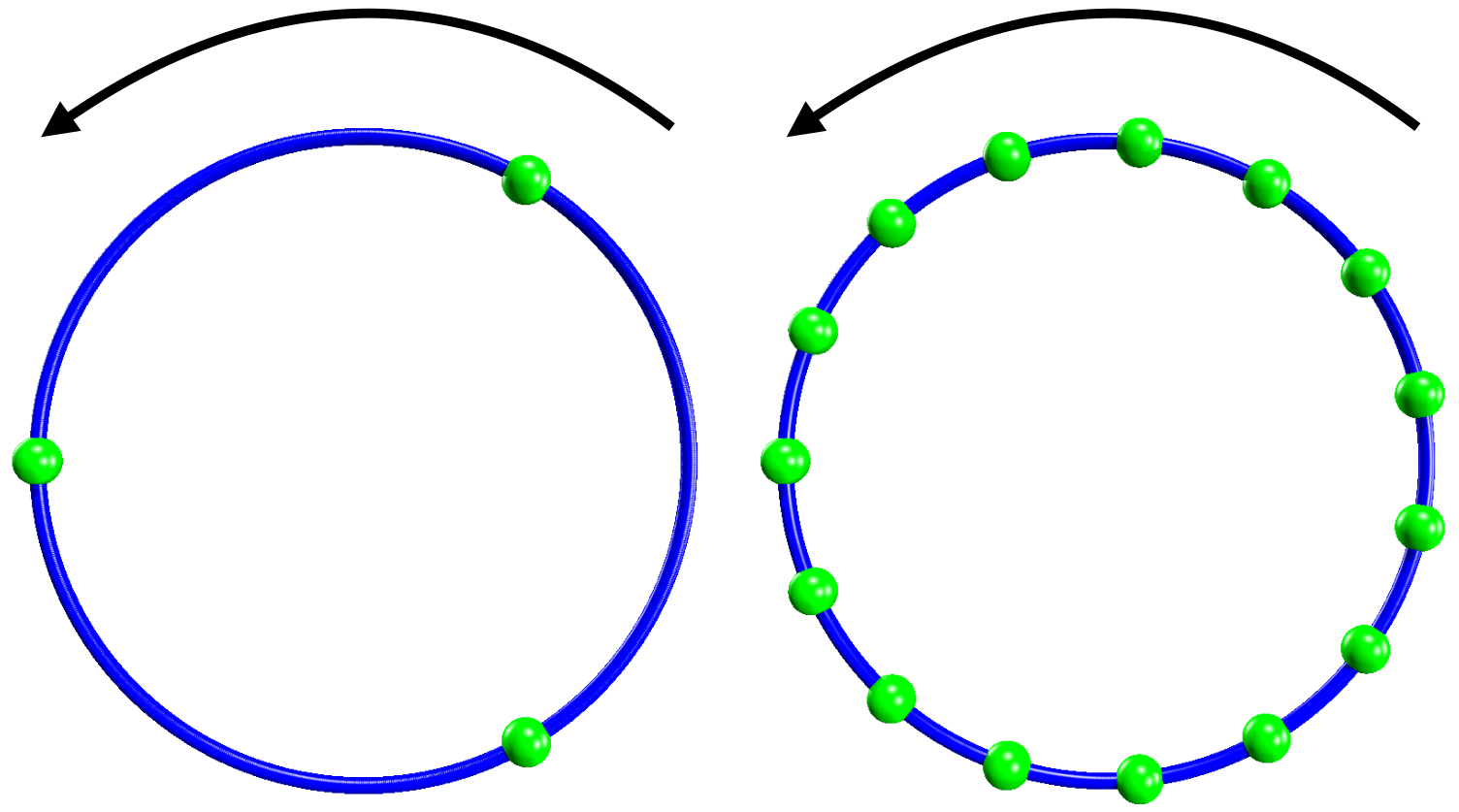}
\vspace{-.3cm}
\caption{\textbf{Circular Choreographies:} 
In this figure green spheres represent masses and the blue 
curve illustrates the circular choreography followed by the 
bodies.  The left frame illustrates a ``snap shot''
of the equilateral triangle solution of Lagrange.   At each instant the three bodies 
are located at the vertices of an equilateral triangle, which rigidly rotates 
with constant angular velocity.  When the masses are equal the triangle revolves 
about the center of mass of the three bodies.  
The right frame illustrates a 15 body circular choreography, where at each moment
the bodies are located at the vertices of a regular $15$-gon.  
The snapshot rotates the center of mass with constant angular velocity.  
Note that if we change to co-rotating coordinates (rotating coordinate
frame origin at the center of mass and angular velocity matching the angular velocity 
of the triangle/polygon) then the triangle/polygon represents an equilibrium configuration 
in rotating coordinates.  
 }
\label{fig:Polygons}
\end{figure}

\begin{figure}[t!]
\centering
\vspace{\baselineskip}
\includegraphics[width=1 \textwidth]{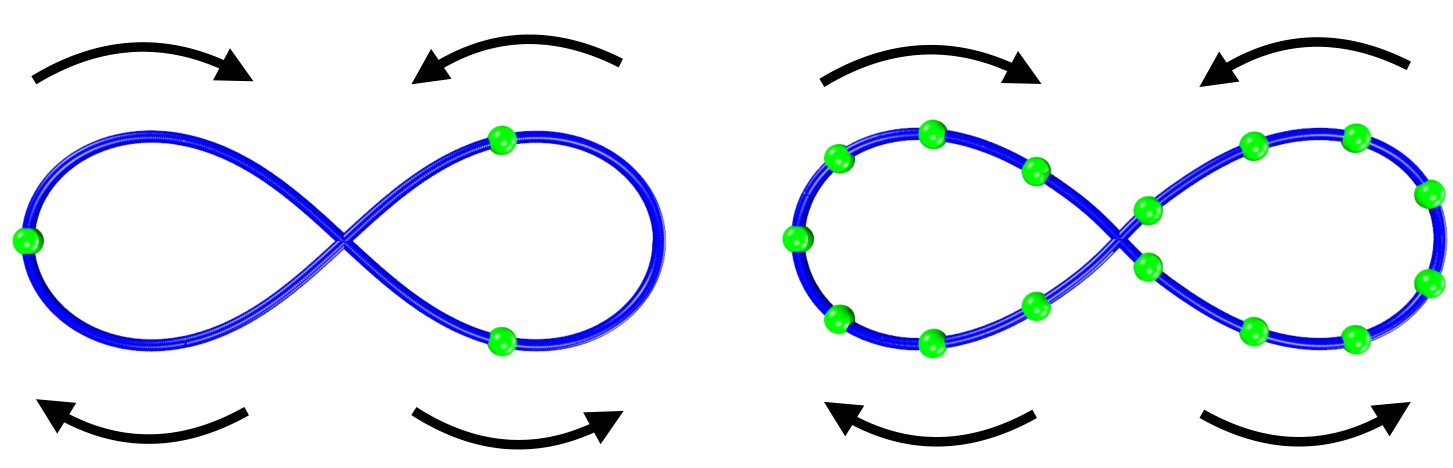}
\vspace{-.3cm}
\caption{\textbf{Figure Eight Choreographies:} just as in 
Figure \ref{fig:Polygons}, green spheres represent the massive bodies and the blue 
curve illustrates the path of the choreography orbit.  The left frame illustrates a ``snap shot''
of the three body eight, and the arrows indicate the motion along the various
segments of the curve.  The right frame is similar, illustrating a 15 body eight.}
\label{fig:Eights}
\end{figure}

It is a fundamental geometric property of conservative systems that periodic orbits occur in 
one parameter families, or tubes, smoothly parameterized by energy/frequency.  
We say that two periodic orbits are in the same continuation class if 
one can be reached from the other by continuous variation of the energy.
Note that, because of bifurcations, the global geometry of a continuation
class is not a single tube but rather a ``tree'', possibly with many branches.   

The literature discussed in the preceding paragraphs makes it
clear that the three body problem admits at least two 
distinct choreography solutions: the equilateral triangle and the eight.
Moreover, the symmetry group of the eight choreography is a $12^{th}$-order
subgroup of the symmetry group of the equilateral triangle.
Then a natural question is:  are these co-existing choreographies 
are related by continuation?
Indeed, one finds in the $2010$ M\'{e}moire D'Habilitation of Jacques F\'{e}joz
\cite{fijozHabilatation} the following recollection regarding Christian Marchal: 
that ``in 1999, when he (Marchal) heard about the choreographic figure-eight 
solution of Chenciner-Montgomery, he at once imagined that the eight could 
be the unknown end of $P12$.''  Here $P12$ is an out of plane family of periodic 
orbits related to the equilateral triangle of Lagrange and  discussed in somewhat 
more detail below.  
These remarks are formalized as follows.

\begin{conjecture} 
[Marchal's Conjecture]\label{theConjecture} 
The three body equilateral triangle of Lagrange and the 
three body figure eight are in the same continuation class.
\end{conjecture}

\noindent The conjecture  appeared also in the 2005 paper
of Chenciner, F\'{e}joz, and Montgomery.  See also the lecture notes of 
Chenciner \cite{MR2446248}.  Careful numerical calculations
supporting Marchal's conjecture are found in the $2008$ work of 
Wulff and Schebesch \cite{MR2429679}
on numerical continuation of relative periodic orbits in Hamiltonian systems.

We remark that Conjecture \ref{theConjecture} should be 
regarded as something much more than a mathematical curiosity. In the 
precise sense of global bifurcation theory, the conjecture concerns the 
question \textit{where does the figure eight choreography come from?}
We hasten to add that, to the best of our knowledge,
the conjecture remains unproved in a completely mathematically rigorous sense. 
Indeed, as is discussed further in \cite{fijozHabilatation}, it appears to be very
difficult to obtain the estimates necessary for a variational proof of
Conjecture \ref{theConjecture}.


We now describe in somewhat more detail 
the $P12$ vertical family of periodic orbits already alluded to above.
The three body equilateral triangle configuration of Lagrange corresponds to
an equilibrium solution of the rotating three body problem, 
and the existence of an attached family of
vertical Lyapunov periodic orbits is established
by Chenciner and F\'{e}joz in \cite{ChFe07}. 
The authors of \cite{ChFe07} compute a normal 
form at the Lagrange relative equilibrium and deduce that there
is a unique  (up to symmetries) bifurcating spatial family: the so
called $P12$ family. The family is important because, as long as it varies
continuously with respect to the frequency parameter, it provides 
--upon returning to the inertial frame--
a dense set of choreographic solutions to the three body problem. 
Continuity with respect to frequency was further established 
in \cite{ChFe08} for frequencies close to the Lagrange triangle.


Using an approach based on equivariant bifurcation theory, Garc\'{i}a-Azpeitia and Ize in 
\cite{GaIz11,GaIz13} studied in rotating coordinates the global existence
of the vertical Lyapunov family arising from the Lagrange triangle. The term
\textit{global} means here that the family forms a continuum (``tube'') in an appropriate
Sobolev space of normalized $2\pi$ periodic solutions. The continuum is
parameterized by frequency and terminates in one of the following
alternatives: the Sobolev norm of the orbits in the family tends to infinity,
the period of the orbits tends to infinity, the family ends in an orbit with
collision, or the family returns to another equilibrium solution.
Without additional information it is not possible
to know which alternative actually occurs, but this vertical 
Lyapunov family makes a good candidate for exploring the continuation
from the equilateral triangle to the eight because of the fact that it 
gives rise to choreography solutions.

What is more, the geometric picture just described extends naturally to any odd number of 
bodies.  The regular $n$-gon solution of Hoppe 
\cite{nGonPaper} (mentioned above) provides a circular choreography, and hence 
a relative equilibrium solution in rotating coordinates, for any number of bodies. 
The $n$-gon equilibrium in the rotating $n$-body problem
has always an attached vertical family of Lyapunov periodic 
orbits  \cite{GaIz11,GaIz13}.  For any even number of bodies the existence of a figure eight 
choreography can be ruled out via symmetry considerations 
(it would result in a finite time collision).  Yet for an odd numbers of bodies 
it is physically possible to have a figure eight, and indeed one finds numerical
evidence supporting the existence of $5$ and $19$ body eights in the classic 
work of Sim\'{o} \cite{Si00}.  See also Ferrario and Terracini \cite{FeTe04},
and note that the right frame of Figure \ref{fig:Eights} illustrates
a numerically computed 15-body eight.   Then one can ask: 
\textit{for which odd numbers $n$ are the $n$-gon and $n$-body eight in 
the same continuation family}?

The study by Calleja, Doedel, and Garc\'{i}a-Azpeitia 
\cite{CaDoGa18} casts additional light on the question.
In that reference the authors explore the behavior of 
the vertical family for different numbers of
bodies using numerical continuation methods. 
Of particular importance to the present study, 
the authors of  \cite{CaDoGa18} discovered a  
numerical continuation leading from the $7$-gon to the $7$-body eight. 
The continuation passes through the vertical 
Lyapunov (or $P12$) family, but also involves a symmetry breaking bifurcation
from this family.
The occurrence of bifurcations helps to explain the difficulty in applying
variational methods.

Combining the three body numerical continuations 
from \cite{MR2429679}
with the five body numerical continuations from \cite{CaDoGa18}
an interesting picture begins to emerge.  These studies suggest 
the possibility that, for odd 
numbers of bodies, continuation from the $n$-gon to the eight
may be the rule rather than the exception. 
This question motivates the present work.

We exploit the functional analytic framework for studying $n$-body choreographies
put forward by the authors of the present study in the recent work \cite{ourTorusKnots}. 
Our approach explicitly incorporates the symmetries and reduces the $n$-body 
choreography problem to a system of six scalar delay differential equations (DDE) describing 
the location of one of the bodies.  The idea is that
$n$-bodies on a choreography swap locations with one another after a fixed
fraction of the period, so that the gravitational force exerted by body $j$ on
body $k$ can be rewritten in terms of a force exerted on body $k$ \textit{by
itself} after an appropriate time shift -- hence the delay. Under explicit number theoretic
conditions on the frequency, periodic solutions of the rotating DDE provide
choreography orbits of the Newtonian $n$-body problem back in the inertial
reference frame.

Another notable component of the present work is that we incorporate the
theoretical insights \cite{GaIz11,GaIz13} into our numerical continuation framework.
More precisely, we exploit the first order description of the vertical
Lyapunov family given in the reference just cited to ``find our way out'' of the
high dimensional kernel caused by the high order resonances at the triangle.
In fact the results of \cite{GaIz11,GaIz13} apply to any number of bodies at
relative equilibrium on the $n$-gon. This leads to a general numerical
procedure for starting the continuation of the vertical Lyapunov family 
for any number of bodies,
and allows us to explore the continuation branch in
an automatic fashion.  The numerical explorations to be presented in the 
remainder of the present work suggest the following conjecture.

\begin{conjecture}
[Generalized Marchal's Conjecture]\label{ourConjecture} 
For any odd number of bodies, the $n$-gon choreography and the 
$n$-body figure eight are in the same continuation class.
\end{conjecture}
 
We present numerical continuation results
for every odd number of bodies from $n = 3$ to $n = 15$,
as evidence in support of the conjecture.  Again, the three and 
fifteen body eights at the end of the continuation
 are illustrated in  Figure \ref{fig:Eights}.
It is also very interesting to report that in each case
the qualitative features of the continuation are the same.
In the symmetrized version of the problem, 
the $n$-body figure eight occurs always after a 
single axial bifurcation from the vertical Lyapunov family associated 
with the regular $n$-gon.

\begin{remark}[DDEs versus ODEs] \label{rem:dde_ode}
In one sense passing to a system of DDEs provides a dramatic reduction in the
dimension of the problem, as we obtain a six scalar equations describing a single body 
instead of a system $6n$ scalar equations for all $n$ bodies.
On the other hand, the appearance of delays
in the problem could be viewed as a major technical disadvantage.
This is because initial
value problems for DDEs lead to infinite dimensional complications, while
the $n$-body problem is inherently finite dimensional.
Nevertheless, after restricting ones attention to periodic solutions and
projecting into Fourier space, both the delay and the differential operators
are reduced to diagonal operations in the space of complex Fourier coefficients.
Thanks to this observation it is the case that, after projecting to Fourier
space, studying a periodic orbit of a system of $6$ DDEs is in principle no
more difficult than studying a periodic orbit in a system of $6$ ODEs.
Moreover, explicitly incorporating the choreographic symmetries into the 
problem results in fewer bifurcations along the periodic branch than would 
be encountered if we continued the vertical family in the full rotating $n$-body 
problem.  To put it another way, in the $n$-body problem a periodic orbit 
bifurcating from a choreography need not be a choreography.  While 
in the DDE, we only see bifurcations that result in new branches containing dense sets of choreographies.
\end{remark}

The remainder of the paper is organized as follows. In 
Section \ref{sec:ddeFormulation}, we review the functional analytic 
formulation of the $n$-body choreography problem developed in 
\cite{ourTorusKnots}. In Section \ref{sec:continuation} we 
discuss the results of a number of numerical continuations, where
we examine (in Section \ref{sec:stability}) the stability of the 
orbits. Finally, in Section \ref{sec:conclusions} we discuss the prospects
for applying computer-assisted methods of proof to the 
generalized conjecture for finite numbers of bodies.

\section{Functional analytic formulation of the \boldmath$n$\unboldmath-body choreography problem}
\label{sec:ddeFormulation}

We describe the movement of the $n$ bodies in a rotating frame with frequency
$\sqrt{s_{1}}$, where%
\begin{equation}
s_{1} \bydef \frac{1}{4}\sum_{j=1}^{n-1}\frac{1}{\sin(j\zeta/2)},\qquad\zeta
\bydef \frac{2\pi}{n}.
\end{equation}
We assume that the masses of the $n$ bodies are equal to $1$. Thus the
$2\pi/\omega$-periodic solutions of the $n$-body problem in the rotating frame
are solutions in the inertial frame of the form
\[
q_{j}(t)=e^{\sqrt{s_{1}}t\bar{J}}u_{j}(\omega t), \qquad j=1,\dots,n
\]
where $u_{j}(t)$ are $2\pi$-periodic functions and $\bar{J}=J\oplus0$ with $J$ the usual
symplectic matrix in $\mathbb{R}^{2}$. Therefore Newton equations in the
coordinates $u_{j}$ read as%
\begin{equation} 
\left(  \omega\partial_{t}+\sqrt{s_{1}}\bar{J}\right)  ^{2}u_{j}=\nabla
_{u_{j}}U=-\sum_{i=1(i\neq j)}^{n}\frac{u_{j}-u_{i}}{\left\Vert u_{j}%
-u_{i}\right\Vert ^{3}}, \label{Ne}%
\end{equation}
where $U$ is the potential energy%
\[
U=\sum_{i<j}\frac{1}{\left\Vert u_{j}-u_{i}\right\Vert }.
\]
Actually, the frequency of rotation is chosen to be$\sqrt{s_{1}}$ such that
the $n$-polygon comprised of $n$ bodies on the unit circle, $a_{j} \bydef (\cos
j\zeta,\sin j\zeta,0)$ for $j=1,...,n$, is an equilibrium solution of
equations (\ref{Ne}), for instance see \cite{GaIz11}.

Set $u=(u_{1},...,u_{n}),$ $a=(a_{1},...,a_{n})$ and $\mathcal{\bar{J}%
} \bydef \bar{J}\oplus...\oplus\bar{J}.$ The linearization of the equation (\ref{Ne})
at the polygonal equilibrium $a$ is
\begin{equation}
\left(  \omega\partial_{t}+\sqrt{s_{1}}\mathcal{\bar{J}}\right)  ^{2}%
u=D^{2}U(a)u.\label{LNe}%
\end{equation}
As a particular consequence of the results obtained in \cite{GaIz13}, we have
that $u(t)=\emph{Re\,(}e^{it}w_{k})$ is a periodic solution of the
linearized system (\ref{LNe}) with frequency $\omega=\sqrt{s_{k}}$ for $k=1,...,n-1$, where
\[
s_{k}\bydef\frac{1}{4}\sum_{j=1}^{n-1}\frac{\sin^{2}(kj\zeta/2)}{\sin^{3}%
(j\zeta/2)}%
\]
and $w_{k} \bydef \left(  w_{k}^{1},...,w_{k}^{n}\right)  $ is a complex vector with
components $w_{k}^{j}=(0,0,e^{j(ik\zeta)})$. This implies that the linearized
system at the polygon (\ref{LNe}) has periodic solutions with frequency
$\sqrt{s_{k}}$. However, since $s_{k}=s_{n-k}$ for $k\in\lbrack1,n/2)\cap
\mathbb{N}$, then the set of solutions with the frequency $\omega=\sqrt{s_{k}%
}$ are in $1:1$ resonance. Furthermore, for the case $\omega=\sqrt{s_{1}}$
there are extra resonances with other frequencies corresponding to planar
components, see \cite{GaIz13} and \cite{ChFe08} for details.

In \cite{ChFe08}, it is proven that there are families of periodic solutions that persist near the polygonal equilibrium for the  $1:1$ resonance frequencies $\omega=\sqrt{s_{k}}$ 
using Weinstein-Moser theory.
Actually,
in \cite{GaIz13} it is proven that these families form a global continuous branch
of solutions (vertical Lyapunov families) with symmetries%
\begin{equation}
u_{j}(t)=e^{\ j\bar{J}\zeta}u_{n}(t+jk\zeta)\text{.}\label{Sy}%
\end{equation}
The existence of a dense set of choreographies in
the vertical Lyapunov families was first pointed out in \cite{ChFe08}.
Later on, in \cite{CaDoGa18} was observed that if $p$ and $q$ are relatively
prime such that
\begin{equation}
kq-p\in n\mathbb{Z},\label{Di}%
\end{equation}
an orbit in the vertical Lyapunov family with the symmetries of (\ref{Sy})
having frequency $\omega=\sqrt{s_{1}}p/q$ is a \emph{simple choreography} in
the inertial reference frame.  Since the
set of number$\sqrt{s_{1}}p/q$ with $p$ and $q$ satisfying the diophantine
equation (\ref{Di}) is dense, if the frequency $\omega$ varies continuously
along the Lyapunov family varies, then there are infinitely many simple
choreographies in the inertial frame.

\begin{remark}
\label{Rem}A consequence of Proposition 3 in \cite{ourTorusKnots} is that
if $u_{n}(t)$ is a solution in the axial family with $p$ and $q$ satisfying (\ref{Di}) and its orbit does not wind around the $z$-axis, then the choreography winds 
on the surface of a toroidal manifold with
winding numbers $p$ and $q$, \textit{i.e.}, the choreographic path is a
$(p,q)$-torus knot. Since the figure eight is a singular $(2,1)$-torus knot, according to this
principle we look for a orbit with $(p,q)=(2,1)$, i.e., our target frequency
$\omega$ is
\[
\omega/\sqrt{s_{1}}=2.
\]
The condition (\ref{Di}) becomes that $k-2\in n\mathbb{Z}$. Thus the only natural match to
find the $(2,1)$-torus knot (figure eight) is the branch with $k=2$. We confirm numerically that branches with $k=2$  effectively contain the figure eight choreographies.
\end{remark}

\subsection{The system reduced by symmetries}

The purpose of this section is to build a self contained setting to present a
systematical approach to obtain numerical computations of the periodic solutions arising from the polygonal
relative equilibrium of the $n$-body problem. In a subsequent paper we plan to present  
rigorous validations of these families. We proceed by imposing the symmetries
(\ref{Sy}) in the system of equations (\ref{Ne}), i.e. the system of equations is reduced to the
following single equation with multiple delays for the $n$-th body
\begin{equation}
\mathcal{G}(u;\omega)\,\overset
{\mbox{\tiny\textnormal{\raisebox{0ex}[0ex][0ex]{def}}}}{=}\,\left(
\omega\partial_{t}+\sqrt{s_{1}}\bar{J}\right)  ^{2}u+G(u)=0,\label{DDE}%
\end{equation}
where $u(t)\in\mathbb{R}^{3}$ is the position of the $n$-th body and $G$
is the nonlinearity%
\[
G(u)=\sum_{j=1}^{n-1}\frac{u-e^{\ j\bar{J}\zeta}u(t+jk\zeta)}{\left\Vert
u-e^{\ j\bar{J}\zeta}u(t+jk\zeta)\right\Vert ^{3}}.
\]

The polygonal equilibrium $a$ has the $n$-th body in the position%
\[
u_{0}=\left(  1,0,0\right)  \text{,}%
\]
i.e. $u_{0}$ is an equilibrium of equation (\ref{DDE}). The linearization of
$\mathcal{G}$ at $u_{0}$ is
\[
D\mathcal{G}(u_{0};\omega)=\left(  \omega\partial_{t}+\sqrt{s_{1}}\bar
{J}\right)  ^{2}u+DG(u_{0}).
\]
Based on the previous discussion we see that the $n$-th component 
of $u(t)=\emph{Im\,(}e^{it}w_{k})$, denoted  by
\[
u_{1}=\left(  0,0,\sin
t\right) ,
\]
is in the kernel of $D\mathcal{G}(u_{0};\sqrt{s_{k}})$. Actually, to conclude that
\[
D\mathcal{G}(u_{0};\sqrt{s_{k}})u_{1}=\left(  s_{k}\partial_{t}^{2}%
+DG(u_{0})\right)  u_{1}=0,
\]
we only need to compute that $DG(u_{0})\,u_{1}=s_{k}u_{1}$.
In the following proposition we present a self-contained proof of this fact,

\begin{lemma}
It holds that
\[
DG(u_{0})\,u_{1}=s_{k}u_{1}.
\]

\end{lemma}

\begin{proof}
We start by considering the nonlinear function $G(u_{0}+\sigma_{1}u_{1}^{\ast
})$, where  $u_{1}=Im\,u_{1}^{\ast}=(0,0,\sin t)^{T},$ and 
$u_{0}=(1,0,0)$ and
$u_{1}^{\ast}=(0,0,e^{it})$. We compute the derivative using the formula 
$DG(u_{0})u_{1}=\emph{Im~}\left(  \partial_{\sigma_{1}}G(u_{0}+\sigma_{1}%
u_{1}^{\ast})|_{\sigma_{1}=0}\right)  $. Set$\ $%
\[
c=-\frac{(1-e^{ijk\zeta})^{2}}{2-2\cos(j\zeta)}=e^{ijk\zeta}\frac{\sin
^{2}(jk\zeta/2)}{\sin^{2}(j\zeta/2)}\in\mathbb{C}\text{.}%
\]
We use the Taylor expansion
\[
(1-x)^{-3/2}=1+{\frac{3}{2}}x+{\frac{15}{8}}x^{2}+{\frac{35}{16}}x^{3}+...,
\]
to compute%
\[%
\begin{split}
G\left(
\begin{array}
[c]{c}%
1\\
0\\
\sigma_{1}e^{it}%
\end{array}
\right)   &  =\sum_{j=1}^{n-1}\left(  \frac{1}{2^{3}\sin^{3}(j\zeta/2)\left(
1-c\sigma_{1}^{2}e^{2it}\right)  ^{3/2}}\left(
\begin{array}
[c]{c}%
1-\cos(j\zeta)\\
-\sin(j\zeta)\\
\sigma_{1}e^{it}(1-e^{ijk\zeta})
\end{array}
\right)  \right)  \\
&  =\sum_{j=1}^{n-1}\frac{1}{2^{3}\sin^{3}(j\zeta/2)}\left(  1+{\frac{3}{2}%
}\sigma_{1}^{2}\left(  ce^{2it}\right)  +...\right)  \left(
\begin{array}
[c]{c}%
1-\cos(j\zeta)\\
-\sin(j\zeta)\\
\sigma_{1}e^{it}(1-e^{ijk\zeta})
\end{array}
\right)  .
\end{split}
\]
Therefore, we have that%
\[%
\left.  \partial_{\sigma_{1}}G(u_{0}+\sigma_{1}u_{1}^{\ast})\right\vert
_{\sigma_{1}=0}=\left(  \sum_{j=1}^{n-1}\frac{1-e^{ijk\zeta}}{2^{3}\sin
^{3}(j\zeta/2)}\right)  \left(
\begin{array}
[c]{c}%
0\\
0\\
e^{it}%
\end{array}
\right)  .
\]
Since
\[
\sum_{j=1}^{n-1}\frac{1-e^{i(jk\zeta)}}{2^{3}\sin^{3}(j\zeta/2)}=\sum
_{j=1}^{n-1}\frac{1-\cos kj\zeta}{2^{3}\sin^{3}(j\zeta/2)}=\frac{1}{4}%
\sum_{j=1}^{n-1}\frac{\sin^{2}(jk\zeta/2)}{\sin^{3}(j\zeta/2)}=s_{k}
\]
we obtain that
\[ 
DG(u_{0})\,u_{1}=\emph{Im~}\left(  \partial_{\sigma_{1}}G(u_{0}+\sigma
_{1}u_{1}^{\ast})|_{\sigma_{1}=0}\right)  =s_{k}u_{1}. 
\]
\end{proof}

By imposing the symmetries (\ref{Sy}), the linear operator $D_{u}%
\mathcal{G}(u_{0};\sqrt{s_{k}})$ for $k\in\lbrack2,n/2]\cap\mathbb{N}$ has
only half of the kernel of the $1:1$ resonance of the linearized system
(\ref{LNe}). It is important to mention that the kernel of $D\mathcal{G(}%
u;\sqrt{s_{k}})$ still has a high dimension; the dimension of the kernel of
$D\mathcal{G(}u;\omega)$ is at least $3$ for the periodic solutions. This is due to the existence
of a $3$-dimensional group of symmetries corresponding to $xy$-rotations,
$z$-translations and time shift. In the following section, we use an augmented
system that reduces the dimension of the kernel generated by these symmetries.

\subsection{The augmented system}

The problem with the kernel of $D\mathcal{G}(u_{0},\omega)$ generated by
the symmetries by $xy$-rotations, $z$-translations and time shift, is solved
by augmenting the map in order to isolate the orbits of solutions. In \cite%
{ourTorusKnots} we present the augmented map that also turns the
non-polynomial DDE into a higher dimensional DDE with polynomial
nonlinearities.

The augmented system with polynomial nonlinearities is given by 
\begin{align}
f(u,v)& \bydef \partial _{t}u-v  \label{eq:f_phase_space} \\
g(\lambda ,u,v,w;\omega )& \bydef \omega ^{2}\partial _{t}v+2\omega \sqrt{s_{1}}%
\bar{J}v-s_{1}\bar{I}u+P(u,w)+\lambda _{1}\bar{J}u+\lambda _{2}v+\lambda
_{3}e_{3} \label{eq:g_phase_space}  \\
h(\alpha ,u,v,w)& \bydef \left\{ \partial _{t}w_{j}+w_{j}^{3}\left\langle
v(t)-e^{\ j\bar{J}\zeta }v(t+jk\zeta ),u(t)-e^{\ j\bar{J}\zeta }u(t+jk\zeta
)\right\rangle +\alpha _{j}w_{j}^{3}\right\} _{j=1}^{n-1}, \label{eq:h_phase_space} 
\end{align}%
where $w=\left\{ w_{j}\right\} _{j=1}^{n-1}$, $e_{3}=(0,0,1)$ and $P$ is the
polynomial nonlinearity with delays%
\begin{equation}\label{P-equation}
P(u,w) \bydef \sum_{j=1}^{n-1}w_{j}^{3}\left( u(t)-e^{\ j\bar{J}\zeta }u(t+jk\zeta
)\right) .
\end{equation}%
These equations are supplemented by the Poincar\'{e} sections $\eta (u)\bydef%
(I_{1},I_{2},I_{3})=0$, where 
\begin{equation*}
I_{1}(u) \bydef \int_{0}^{2\pi }u(t)\cdot \bar{J}\tilde{u}(t)~dt,\qquad
I_{2}(u) \bydef \int_{0}^{2\pi }u(t)\cdot \tilde{u}'(t)~dt,\qquad
I_{3}(u) \bydef \int_{0}^{2\pi }u_{3}(t)~dt,
\end{equation*}
where $\tilde{u}$ is a reference  function, and the initial conditions%
\begin{equation*}
\gamma (u,w)\bydef\left\{ w_{j}(0)^{2}\left\Vert u(0)-e^{\ j\bar{J}\zeta
}u(jk\zeta )\right\Vert ^{2}-1\right\} _{j=1}^{n-1}=0.
\end{equation*}%
Actually, in Proposition 4 and 5 of \cite{ourTorusKnots} it is proved that the
solutions of $\mathcal{G}(u,\omega )=0$ are equivalent to the solution of
the augmented system of equations%
\begin{equation*}
F(x;\omega ) = \left( \eta ,\gamma ,f,g,h\right) (x;\omega )=0,\qquad x\bydef%
(\lambda ,\alpha ,u,v,w)\text{.}
\end{equation*}

We set the equilibrium for the augmented system as 
\begin{equation}  \label{eq:equilibrium_polygon}
x_{0}=\left( 0,0,u_{0},0,w_{0}\right) ,
\end{equation}
where $w_{0}=\left\{ w_{j,0}\right\} _{j=1}^{n-1}$ with 
\begin{equation*}
w_{j,0}=\frac{1}{\left\Vert (1,0,0)-(\cos j\zeta,-\sin
j\zeta,0)\right\Vert }=\frac{1}{2\sin j\zeta/2}\text{.%
}
\end{equation*}
Using that $u_{0}$ is a steady solution of $\mathcal{G}$ it is not difficult
to see that $x_{0}$ is a steady solution of the augmented system $%
F(x_{0};\omega)=0$ for all $\omega.$

Set 
\begin{equation}  \label{eq:tangent_at_polygon}
x_{1} \bydef (0,0,u_{1},v_{1},0)
\end{equation}
with $u_{1}=(0,0,\sin t)$ and $v_{1}=(0,0,\cos t)$. In the following
proposition we prove that $x_{1}$ is the natural extension of the element of
the kernel of the augmented map $F$.
Notice that from the definition of $P(u,w)$ in \eqref{P-equation}, we have that,
\begin{align}\label{partial_P}
\partial _{u}P(u_{0},w_{0})u_{1}=&\sum_{j=1}^{n-1}w_{j,0}^{3}\left( u_{1}(t)-e^{\ j\bar{J}\zeta
}u_{1}(t+jk\zeta )\right)
= \sum_{j=1}^{n-1}\frac{\sin^{2}(jk\zeta/2)}{\sin^{3}(j\zeta/2)} u_1
= DG(u_{0})u_{1}\text{.}
\end{align}

\begin{proposition}
\label{prop:kernel_at_polygon} It holds that 
\begin{equation*}
DF(x_{0};\sqrt{s_{k}})x_{1}=0\text{.}
\end{equation*}
\end{proposition}

\begin{proof}
We have that%
\begin{equation*}
DF(x_{0})x_{1}=\left( 
\begin{array}{c}
\partial _{u}\eta (x_{0})u_{1} \\ 
\partial _{u}\gamma (x_{0})u_{1} \\ 
\partial _{u}f(x_{0})u_{1}+\partial _{v}f(x_{0})v_{1} \\ 
\partial _{u}g(x_{0})u_{1}+\partial _{v}g(x_{0})v_{1} \\ 
\partial _{u}h(x_{0})u_{1}+\partial _{v}h(x_{0})v_{1}%
\end{array}%
\right) .
\end{equation*}
At the $n$-gon, we set as reference functions $\tilde{u}=u_{1}$ and $\tilde{u}'=u_1'=v_{1}$, so
\begin{equation*}
\partial _{u}\eta
(u_{0})u_{1}=\int_{0}^{2\pi }\left( u_{1}\cdot \bar{J}u_{1},u_{1}(t)\cdot
v_{1}(t),\sin t\right) =0.
\end{equation*} 
For the derivative of $\gamma$,  we have, 
\begin{equation*}
\partial _{u}\gamma (x_{0})u_{1}=\left\{ 2w_{j,0}(0)^{2}\left\langle
u_{0}(0)-e^{\ j\bar{J}\zeta }u_{0}(jk\zeta ),u_{1}(0)-e^{\ j\bar{J}\zeta
}u_{1}(jk\zeta )\right\rangle \right\} _{j=1}^{n-1}=0\text{.}
\end{equation*}%
The derivative of $h$ is,
\begin{equation*}
\partial _{u}h(x_{0})u_{1}=\left\{ w_{j,0}^{3}(t)\left\langle v_{0}(t)-e^{\ j%
\bar{J}\zeta }v_{0}(t+jk\zeta ),u_{1}(t)-e^{\ j\bar{J}\zeta }u_{1}(t+jk\zeta
)\right\rangle \right\} _{j=1}^{n-1}=0\text{,}
\end{equation*}%
since $v_0=0$.
Now, $v_{1}=(0,0,\cos t)$ implies that
\begin{equation*}
\partial _{v}h(x_{0})v_{1}=\left\{ w_{j,0}^{3}(t)\left\langle v_{1}(t)-e^{\ j%
\bar{J}\zeta }v_{1}(t+jk\zeta ),u_{0}(t)-e^{\ j\bar{J}\zeta }u_{0}(t+jk\zeta
)\right\rangle \right\} _{j=1}^{n-1}=0\text{.}
\end{equation*}%
Finally, $\partial _{u}f(x_{0})u_{1}+\partial _{v}f(x)v_{1}=\partial
_{t}u_{1}-v_{1}$, and from \eqref{partial_P} it follows that%
\begin{equation*}
\partial _{u}g(x_{0})u_{1}+\partial _{v}g(x_{0})v_{1}=\omega ^{2}\partial
_{t}v_{1}+2\omega \sqrt{s_{1}}\bar{J}v_{1}-s_{1}\bar{I}u_{1}+\partial
_{u}P(u_{0},w_{0})u_{1}=D_{u}\mathcal{G(}u_{0};\sqrt{s_{k}})u_{1}=0\text{.}
\end{equation*}
\end{proof}


\section{Numerical continuation from the polygon to the eight} \label{sec:continuation}

As mentioned in the introduction (e.g. see Remark~\ref{rem:dde_ode}), our computational approach to choreographies is based on Fourier expansions of the functions $u(t),v(t)$ and $w(t)$ appearing in \eqref{eq:f_phase_space}, \eqref{eq:g_phase_space} and \eqref{eq:h_phase_space}. To compute choreographies, we plug the Fourier expansions 
\begin{align}
\nonumber
u(t)  &  =
\begin{pmatrix}
u_{1}(t)\\
u_{2}(t)\\
u_{3}(t)
\end{pmatrix}
= \sum_{\ell\in\mathbb{Z}}e^{i \ell t}u_{\ell}, \quad u_{\ell} =
\begin{pmatrix}
(u_{1})_{\ell}\\
(u_{2})_{\ell}\\
(u_{3})_{\ell}%
\end{pmatrix}
\\
v(t)  &  =%
\begin{pmatrix}
v_{1}(t)\\
v_{2}(t)\\
v_{3}(t)
\end{pmatrix}
=\sum_{\ell\in\mathbb{Z}}e^{i \ell t}v_{\ell}, \quad v_{\ell} =
\begin{pmatrix}
(v_{1})_{\ell}\\
(v_{2})_{\ell}\\
(v_{3})_{\ell}%
\end{pmatrix}
\label{eq:Fourier_expansions}
\\
w(t)  &  =%
\begin{pmatrix}
w_{1}(t)\\
\vdots\\
w_{n-1}(t)
\end{pmatrix}
=\sum_{\ell\in\mathbb{Z}}e^{i \ell t}w_{\ell}, \quad w_{\ell} =
\begin{pmatrix}
(w_{1})_{\ell}\\
\vdots\\
(w_{n-1})_{\ell}
\end{pmatrix}
\nonumber
\end{align}
in equations \eqref{eq:f_phase_space}, \eqref{eq:g_phase_space}, \eqref{eq:h_phase_space}, the Poincar\'e sections $\eta(u)$ and the initial conditions $\gamma(u,w)$, which leads to a zero finding problem $F(x,\omega)=0$ (still denoted using $F$ and $x$) posed on a Banach product space of geometrically decaying Fourier sequences (see \cite{ourTorusKnots} for more details).

To perform computations, we fix a truncation order $m>0$ and truncate the Fourier series of each component of $u$, $v$ and $w$ to trigonometric polynomials of order  $m-1$. For instance, the function $u_1$ is only represented by the $2m-1$ Fourier coefficients $((u_{1})_{\ell})_{|\ell|<m}$. Similarly for the other components. After truncation, the functions $u$, $v$ and $w$ are represented respectively by $3(2m-1)$, $3(2m-1)$ and $(n-1)(2m-1)$ Fourier coefficients. Adding the unfolding parameters $\lambda \in \C^3$ and $\alpha \in \C^{n-1}$ to the set of unknowns yields a total number of $(n+5)(2m-1)+3+n-1=2m(n + 5) - 3$ of variables. Performing a similar truncation to the functions $f$, $g$ and $h$ defined in \eqref{eq:f_phase_space}, \eqref{eq:g_phase_space} and \eqref{eq:h_phase_space} leads to the {\em finite dimensional projection} $F^{(m)} : \mathbb{C}^{2m(n + 5) - 3} \to \mathbb{C}^{2m(n + 5) - 3}$ which we use to compute numerical approximations of the choreographies. To simplify the presentation, we denote $F = F^{(m)}$ and $N \bydef {2m(n + 5) - 3}$ so that computing a choreography is equivalent to compute $(x,\omega) \in \C^{N+1}$ such that $F(x,\omega) \approx 0$. 

\subsection{Pseudo-Arclength Continuation} \label{sec:PAL_Continuation}

The numerical continuation from the polygon to the figure eight exploits the pseudo-arclength continuation algorithm (e.g. see Keller \cite{MR910499}), and we briefly recall the main idea
behind this approach, which is that the pseudo-arclength is taken as the continuation parameter. 
Then the original continuation parameter, in our case the frequency $\omega$,
is not fixed and instead is left as a variable.  That is 
the vector of variables becomes $X \bydef (x,\omega) \in \C^{N+1}$. Denote by $\mathcal S \bydef \{ X \in \C^{N+1} : F(X)=0\}$ the {\em solution set}.  We aim to compute one dimensional solution curves in 
$\mathcal S$. The process begins with a solution $X_0$ given within a prescribed tolerance. To produce a {\em predictor} (that is a good numerical approximation to feed to Newton's method), we compute first a unit tangent vector to the curve at $X_0$, that we denote $\dot X_0$, which can be computed using the formula
\[
D_X F(X_0) \dot X_0 = \left[ D_x F(X_0) ~~  \frac{\partial F}{\partial \omega}(X_0) \right] \dot X_0 = 0 \in \C^N.
\]
Next fix a {\em pseudo-arclength parameter} $\Delta_s>0$, and set the predictor to be
\[
\hat{X}_1 \bydef \bar X_0 + \Delta_s \dot X_0 \in \C^{N+1}.
\]
Once the predictor is fixed, we {\em correct} toward the set $\mathcal S$ on the hyperplane perpendicular to the tangent vector $\dot X_0$ which contains the predictor $\hat{X}_1$. The equation of this plan is given by
\[
E(X) \bydef (X-\hat X_1) \cdot \dot X_0 = 0.
\]
Then, we apply Newton's method to the new function
\begin{equation} \label{eq:PAL_fonction}
X \mapsto \begin{pmatrix} E(X) \\ F(X) \end{pmatrix}
\end{equation}
with the initial condition $\hat X_1$ in order to obtain a new solution $X_1$ given again within a prescribed tolerance. We reset $X_1 \mapsto X_0$ and start over. See Figure~\ref{fig:PAL_cont} for a geometric interpretation of one step of the pseudo-arclength continuation algorithm. At each step of the algorithm, the function defined in \eqref{eq:PAL_fonction} changes since the plane $E(X)=0$ changes. With this method, it is possible to continue past folds. Repeating this procedure iteratively produces a branch of solutions.

\begin{figure}[h!]
\centering
\vspace{\baselineskip}
\includegraphics[width=.5 \textwidth]{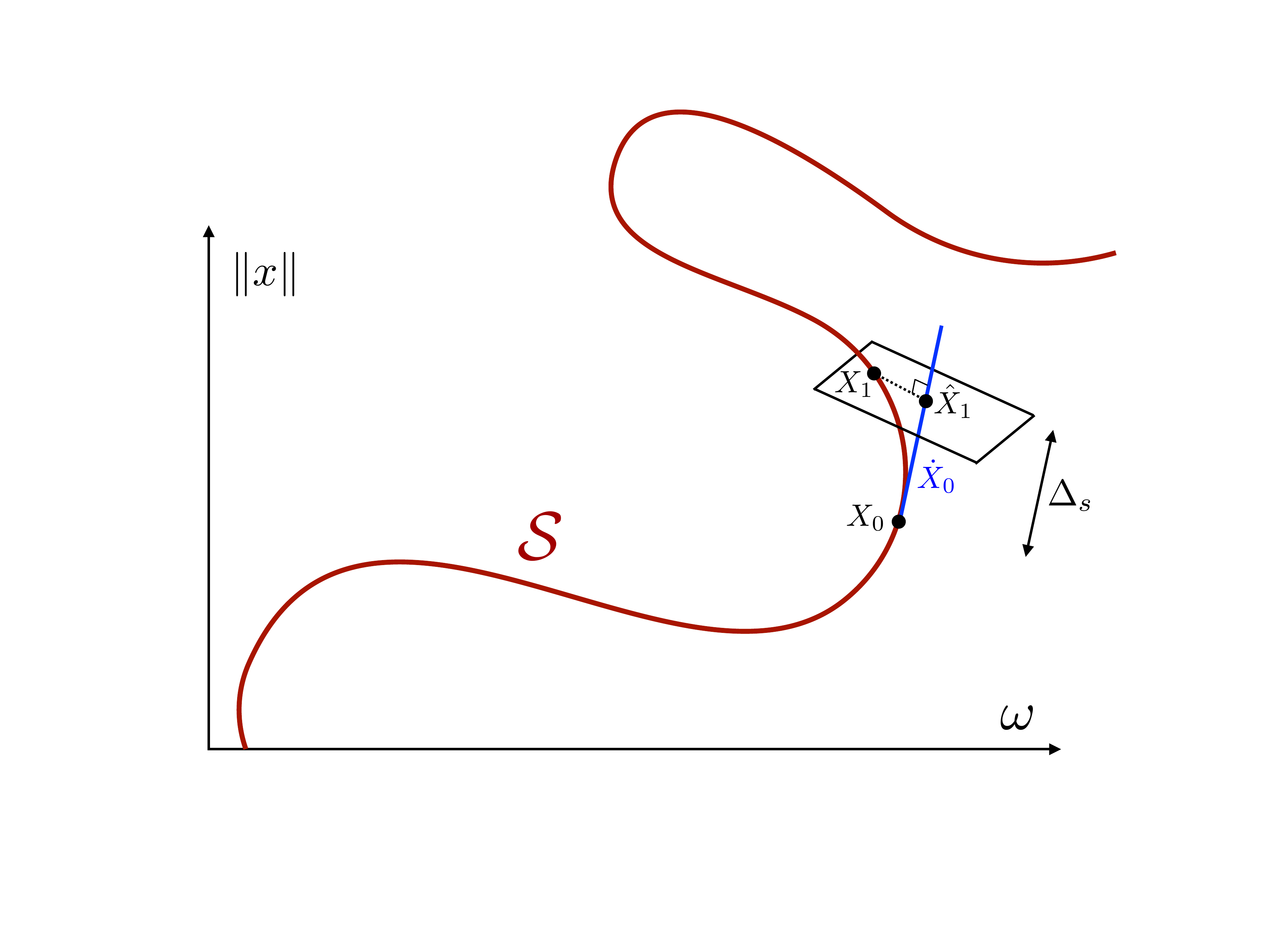}
\vspace{-.3cm}
\caption{Pseudo-arclength continuation.  By following the tangent vector to the 
solution branch and projecting into a suitable perpendicular plane, pseudo-arclength 
continuation is indifferent to fold bifurcations and allows for the following of more complicated
solution curves.}
\label{fig:PAL_cont}
\end{figure}

\subsection{From the polygon to the figure eight} \label{sec:from_polygon_to_eight}

Having introduced the pseudo-arclength continuation, we describe a numerical procedure to bifurcate away from the polygon equilibrium onto an off ($x,y$)-plane spatial family of choreographies, to detect a secondary bifurcation, to perform a branch switching, and finally to reach the figure eight choreography. This process requires a starting point $X_0$ (the polygon equilibrium) and a tangent vector $\dot X_0$ at the polygon.

\subsubsection{Initiating the continuation at the polygon}

Recall the definition of $x_0$ the polygon equilibrium \eqref{eq:equilibrium_polygon} and $x_1$ the vector given by \eqref{eq:tangent_at_polygon}. Abusing slightly the notation, denote $x_0,x_1 \in \C^N$ the corresponding vectors of Fourier coefficients. More explicitly, $x_{0}=\left(  0,0,u_{0},0,w_{0}\right) \in \C^{N}$ is defined by $u_{0}=\left( (\delta_{\ell,0})_{|\ell|<m},0,0 \right) \in \C^{3(2m-1)}$ and for $j=1,\dots,n-1$, $(w_{0})_j= \left( \frac{\delta_{\ell,0}}{2\sin j\pi/n} \right)_{|\ell|<m} \in \C^{2m-1}$. Here, $\delta_{\ell,k}$ denotes the Kronecker delta symbol. Moreover, $x_{1}=(0,0,u_{1},v_{1},0)$ is defined by $u_{1}=(0,0,u^{(3)}_{1}) \in \C^{3(2m-1)}$ and $v_{1}=(0,0,v_{1}^{(3)}) \in \C^{3(2m-1)}$ with $u^{(3)}_{1},v^{(3)}_{1} \in \C^{2m-1}$ given component-wise by
\[
\left( u^{(3)}_{1} \right)_\ell = \begin{cases}
i/2, & \ell = -1 \\
-i/2, & \ell = 1 \\
0, & \text{otherwise}
\end{cases}
\qquad
\text{and}
\qquad
\left( v^{(3)}_{1} \right)_\ell = \begin{cases}
1/2, & \ell = -1 \\
1/2, & \ell = 1 \\
0, & \text{otherwise}.
\end{cases}
\]
Denote $\omega_0 \bydef \sqrt{s_k}$, $X_0 \bydef (x_0,\omega_0)$, and recall Proposition~\ref{prop:kernel_at_polygon}. Then $F(X_0)=0 \in \C^{N}$ and $D_x F(X_0)x_1 =0 \in  \C^{N}$. Denote $\dot X_0 \bydef (x_1,0) \in \C^{N+1}$ and consider the $N \times (N+1)$ dimensional matrix $D_X F(X_0) = [D_x F(X_0) ~ D_\omega F(X_0)]$. Hence, at the polygon $X_0$, 
\[
F(X_0) = 0 \quad \text{and} \quad D_X F(X_0) \dot X_0 = 0,
\]
which provide with the data required to initiate the numerical pseudo-arclength continuation on the problem $F:\C^{N+1} \to \C^N$ as presented in Section~\ref{sec:PAL_Continuation}. We fix the pseudo-arclength parameter (the continuation step size) to be $\Delta_s = 10^{-3}$ and initiate the continuation.

\subsubsection{Switching branches at the secondary bifurcation} 

Along the continuation, we monitor two quantities, namely the sign of the determinant of the derivative of the {\em pseudo-arclength map} given in \eqref{eq:PAL_fonction} and the condition number of the derivative.  If there is a change of sign in the determinant and if the condition number is above a certain threshold (in our case $10^3$), we declare having detected a secondary bifurcation and begin a bisection algorithm to converge to a (bifurcation) point $X_{\rm bif}$ at which the determinant of the derivative of pseudo-arclength map is approximatively zero. In this case, we numerically verify that $\dim\ker(D_XF(X_{\rm bif})) = 2$, and we call $X_{\rm bif}$ a {\em simple} branching point. At $X_{\rm bif}$, there are two solution branches intersecting. Denote by $X_{\rm bif}^{(1)},X_{\rm bif}^{(2)} \in \C^{N+1}$ the two tangent vectors. See Figure~\ref{fig:branch_switching} for a graphical representation of the situation. 
\begin{figure}[h!]
\centering
\vspace{\baselineskip}
\includegraphics[width=.4 \textwidth]{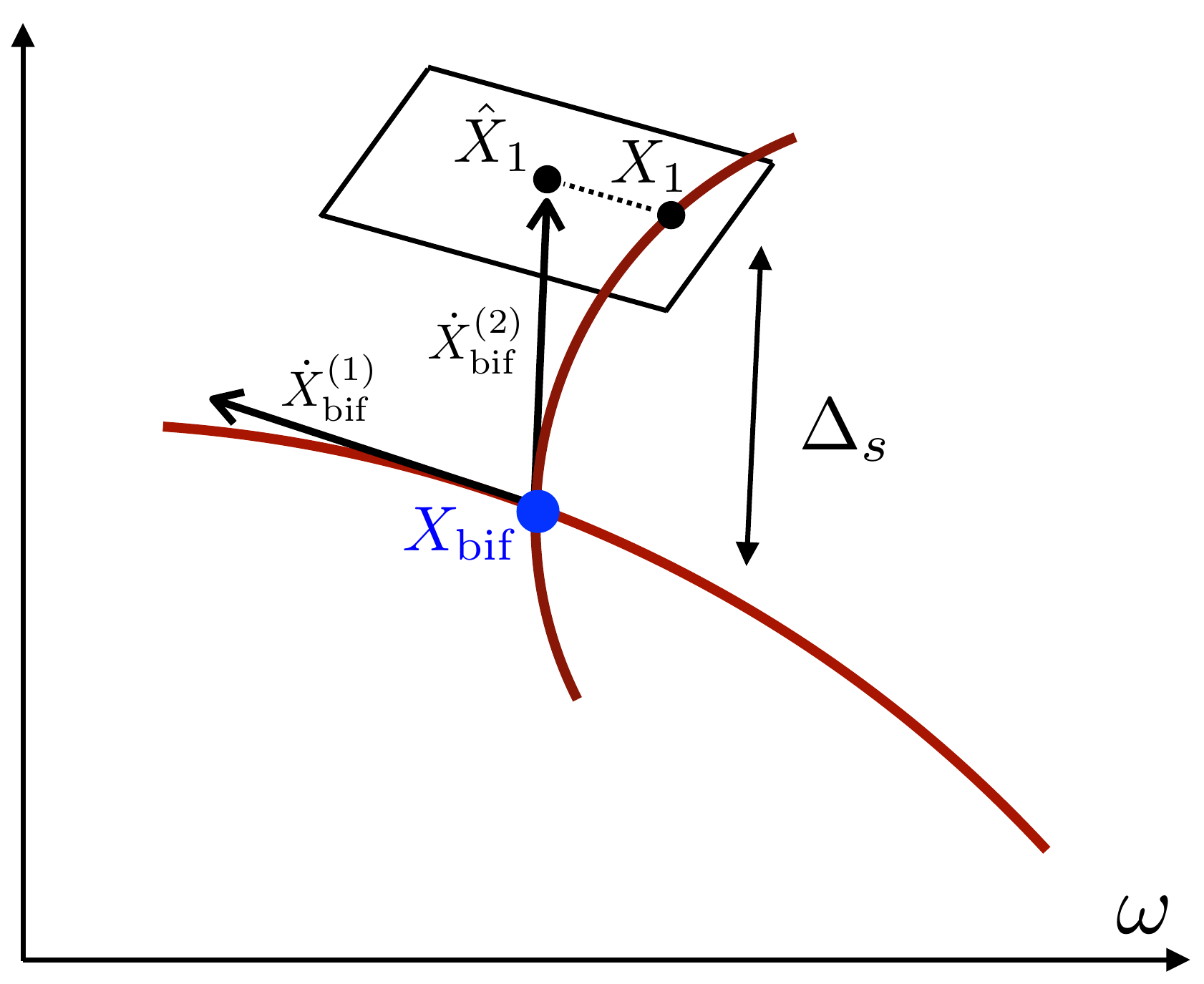}
\vspace{-.3cm}
\caption{Intersection of two solution branches at a simple branching point. Analyzing the 
kernel at the bifurcation facilitates the branch switching.}
\label{fig:branch_switching}
\end{figure}
Note that the exact tangent vectors $X_{\rm bif}^{(1)}$ and $X_{\rm bif}^{(2)}$ are not readily available and computing them accurately would require solving an algebraic bifurcation equation (e.g. see \cite{MR1159608}), which we do not perform here. Instead consider $\phi_1,\phi_2 \in \C^{N+1}$ by any two vectors computed numerically (we use the singular value decomposition of $D_XF(X_{\rm bif})$ to do that) such that $\ker(D_XF(X_{\rm bif})) = \langle \phi_1,\phi_2 \rangle$. Since we expect this secondary bifurcation to be a generic pitchfork bifurcation with respect to the parameter $\omega$, we numerically set 
\[
\hat X_{\rm bif}^{(2)} \bydef (\phi_2)_{N+1} \phi_1  -(\phi_1)_{N+1} \phi_2 \in \ker(D_XF(X_{\rm bif})) \approx X_{\rm bif}^{(2)}
\]
so that it has a zero tangent contribution in the parameter $\omega$. We then set $X_0 = X_{\rm bif}$ and $\dot X_0 = \hat  X_{\rm bif}^{(2)}$, and initiate the numerical pseudo-arclength continuation on the problem $F:\C^{N+1} \to \C^N$ as presented in Section~\ref{sec:PAL_Continuation}.

Since the figure eight is expected to occur (according to Remark \ref{Rem}) at 
\[
\omega = 2 \sqrt{s_1} = \frac{1}{2}\sum_{j=1}^{n-1}\frac{1}{\sin(\pi j/n)},
\]
 we monitor the sign of the function 
$\omega \mapsto 2 \sqrt{s_1} - \omega$ along the continuation on the second branch.
When it changes sign, we fix $\omega = 2 \sqrt{s_1}$ and run Newton's method to obtain the figure eight. 


\subsection{Stability}  \label{sec:stability}

An important question is to consider the stability of choreographic solutions,
and for this we put aside the symmetrized DDE formulation, and 
 return to the original rotating $n$-body problem.
That is, we start with the second order problem defined in Equation \eqref{Ne}.
Let $f\colon \mathbb{R}^{6n} \to \mathbb{R}^{6n}$ denote the corresponding first 
order vector field obtained by appending the velocity variables to the equation, 
and suppose that $\gamma \colon [0, T] \to \mathbb{R}^{6n}$ is a 
periodic solution of the problem.  That is, assume that $T > 0$
and that $\gamma(t)$ solves the ordinary differential equation
$\gamma'(t) = f(\gamma(t))$ for $t \in (0,T)$ with $\gamma(0) = \gamma(T)$. 

The \textit{equation of first variation} is the non-autonomous linear 
matrix 
initial value problem defined by  
\[
M'(t) = Df(\gamma(t)) M(t),  \quad \quad \quad M(0) = \mbox{Id}.
\]
The stability of the 
periodic orbit is determined by the eigenvalues of the monodromy matrix
$M(T)$.  We refer to the number of unstable eigenvalues of $M(T)$
as the \textit{Morse index} of $\gamma$.  

To compute the Monodromy matrix we 
numerically integrate the system of equations 
\begin{align*}
\gamma' &= f(\gamma) \\
M' &= Df(\gamma) M
\end{align*}
over the time interval $[0,T]$, where $T > 0$ is the period of the orbit.
We use a standard Runge-Kutta scheme built into MatLab (the 
standard \verb|rk45|).  
The initial conditions are $\gamma(0) = \gamma_0$ and $M(0) = \mbox{Id}$,
where $\gamma_0$ is a point on the periodic orbit.   The initial conditions 
$\gamma_0$ are obtained exploiting the fact that we have already computed 
the Fourier coefficients of the trajectory of the $n$-th body using the DDE
formulation.  The Fourier series of the trajectories for
the other $n-1$ bodies are recovered 
from the symmetries, and by evaluating the Fourier series we obtain
an appropriate initial condition $\gamma_0 \in \mathbb{R}^{6n}$ on the 
periodic orbit.

\subsection{Numerical results} \label{sec:Examples}

We applied successfully the numerical approach of Section~\ref{sec:from_polygon_to_eight} to the cases $N = 3,5, 7, 9, 11,13$ and $15$ bodies.  Stability is computed as discussed in Section~\ref{sec:stability}.
 Figure~\ref{fig:3_5_bodies}, illustrates the numerical continuation from the equilateral triangle to the figure eight for $N=3$ bodies, and illustrates also the analogous computation for $N=5$ bodies. The Morse indices are also reported along the branches. Similarly, see Figure~\ref{fig:7_9_bodies} for the continuations for $N=7$ and $N=9$ bodies. In the case of  $N=11$, $N=13$ and $N=15$ bodies, we do not report the Morse indices, as the stability of the orbits change too frequently along the continuation 
branch.  But the bifurcation diagrams are given in  Figure~\ref{fig:11_13_15_bodies}.

\begin{figure}[h!]
\centering
\vspace{\baselineskip}
\includegraphics[width=.45 \textwidth]{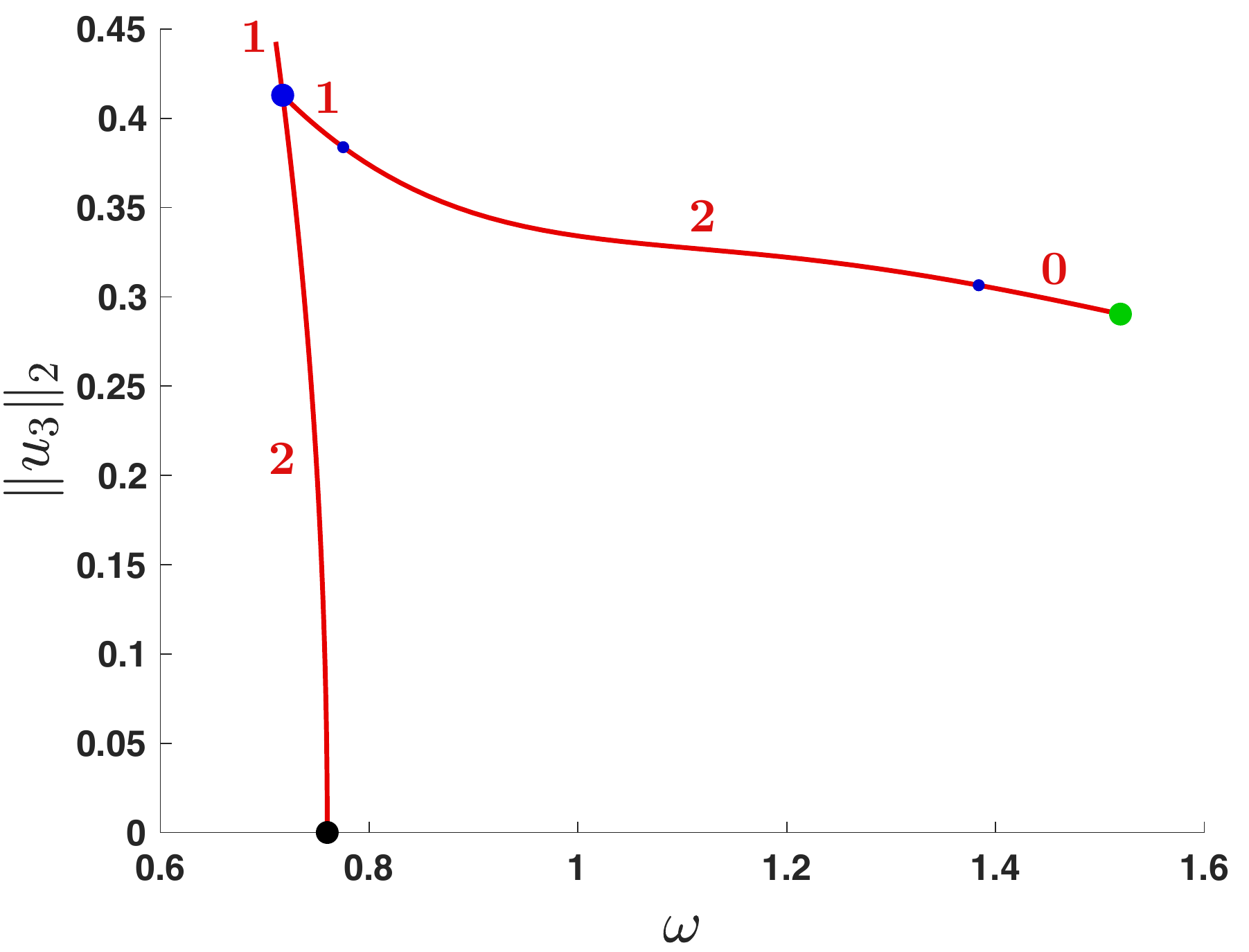}
\includegraphics[width=.45 \textwidth]{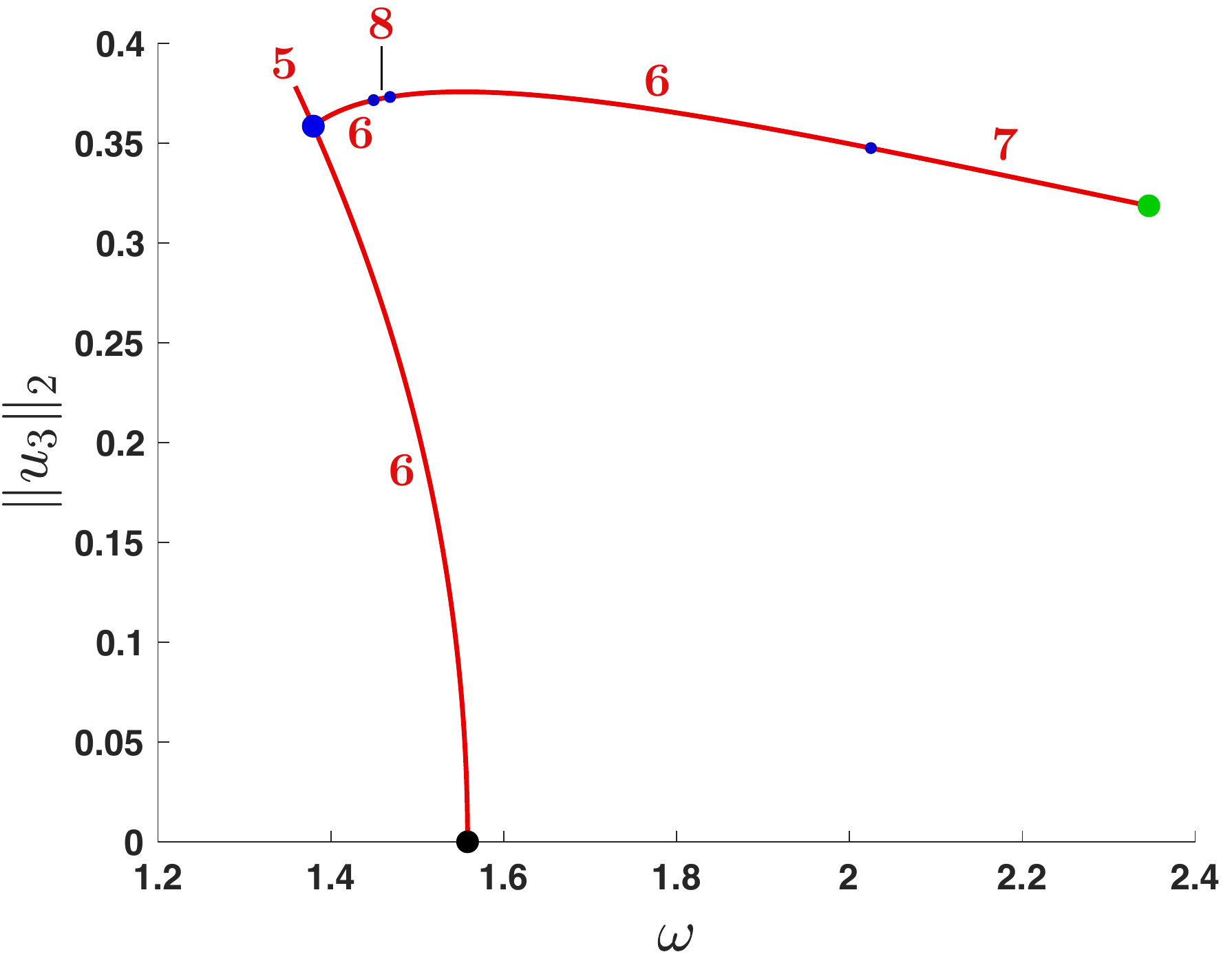}
\vspace{-.3cm}
\caption{Continuation from the triangle (the black dot) to the figure eight (the green dot) for $N=3$ (left) and $N=5$ (right) bodies. The different Morse indices are portrayed along the branches.}
\label{fig:3_5_bodies}
\end{figure}

\begin{figure}[h!]
\centering
\vspace{\baselineskip}
\includegraphics[width=.48 \textwidth]{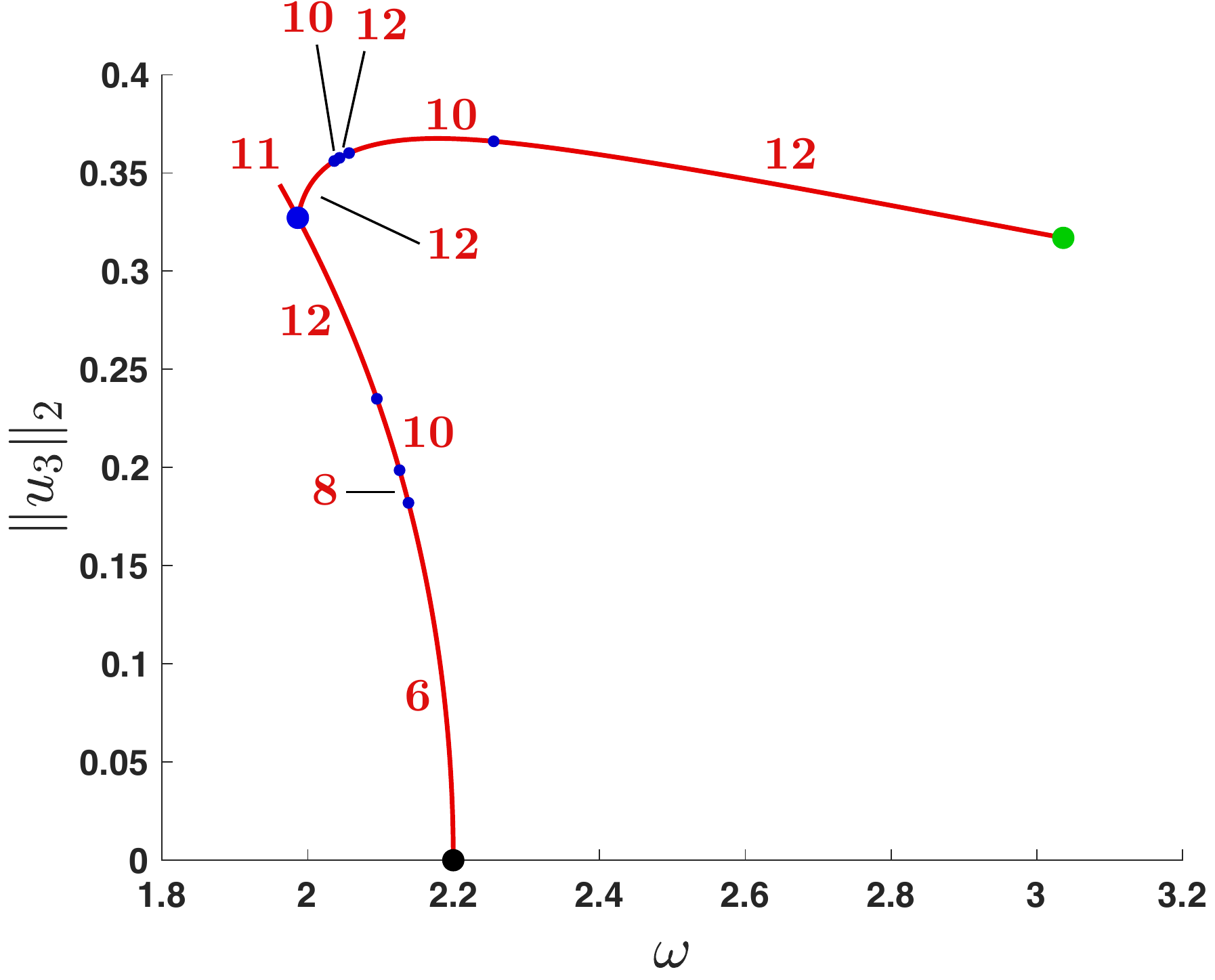}
\includegraphics[width=.48 \textwidth]{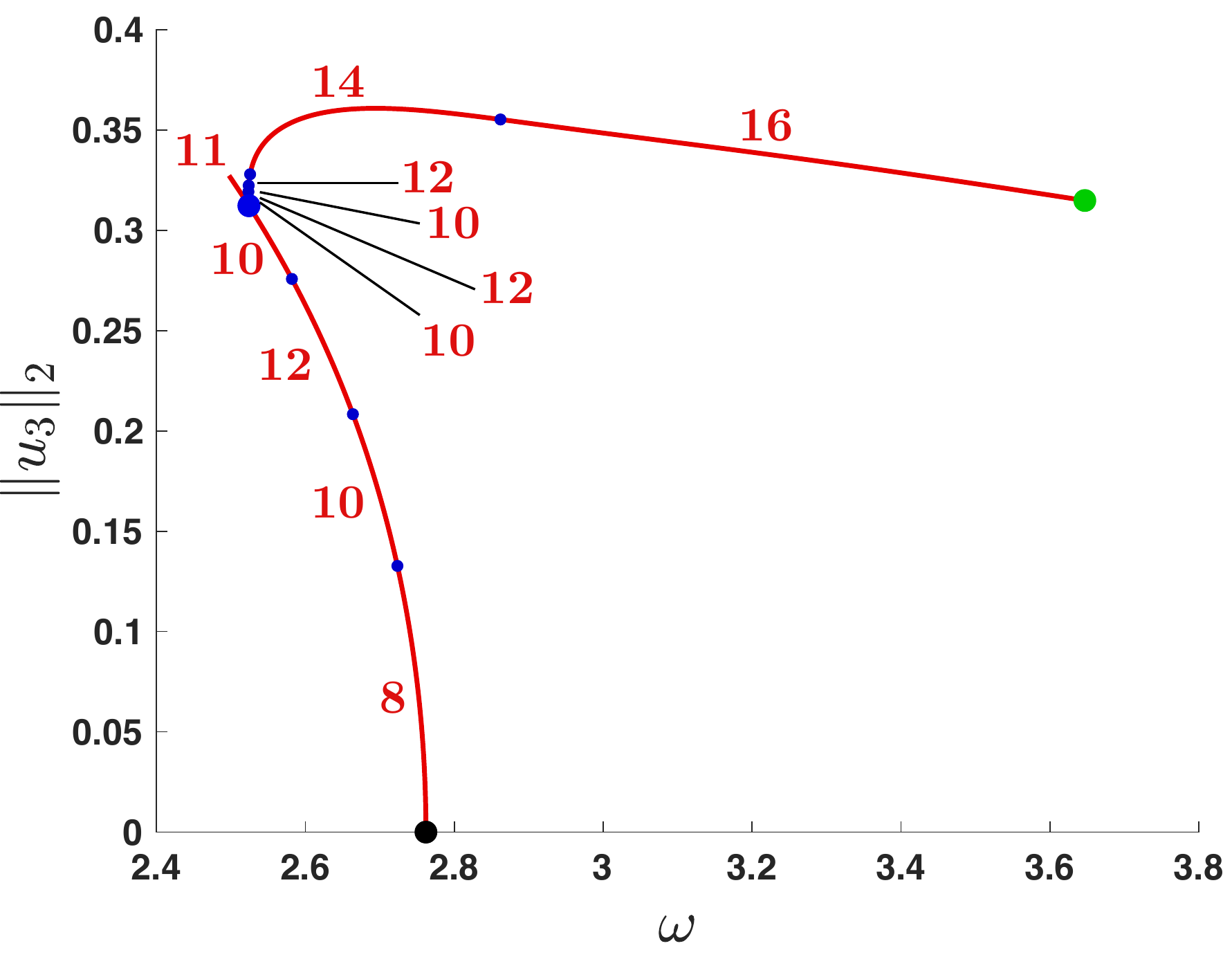}
\vspace{-.3cm}
\caption{Continuation from the triangle (the black dot) to the figure eight (the green dot) for $N=7$ (left) and $N=9$ (right) bodies. The different Morse indices are portrayed along the branches.}
\label{fig:7_9_bodies}
\end{figure}

\begin{figure}[h!]
\centering
\vspace{\baselineskip}
\includegraphics[width=.32 \textwidth]{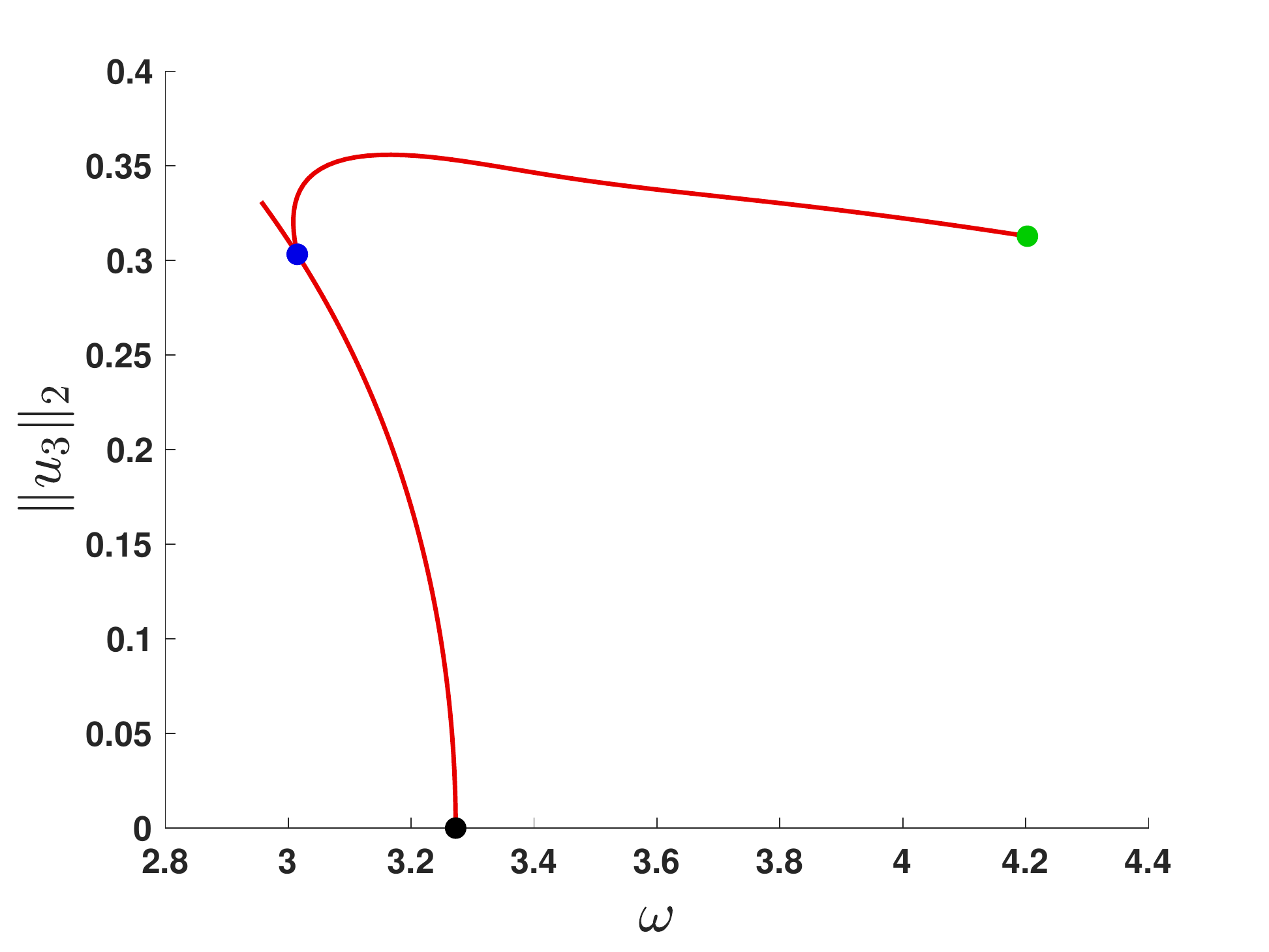}
\includegraphics[width=.32 \textwidth]{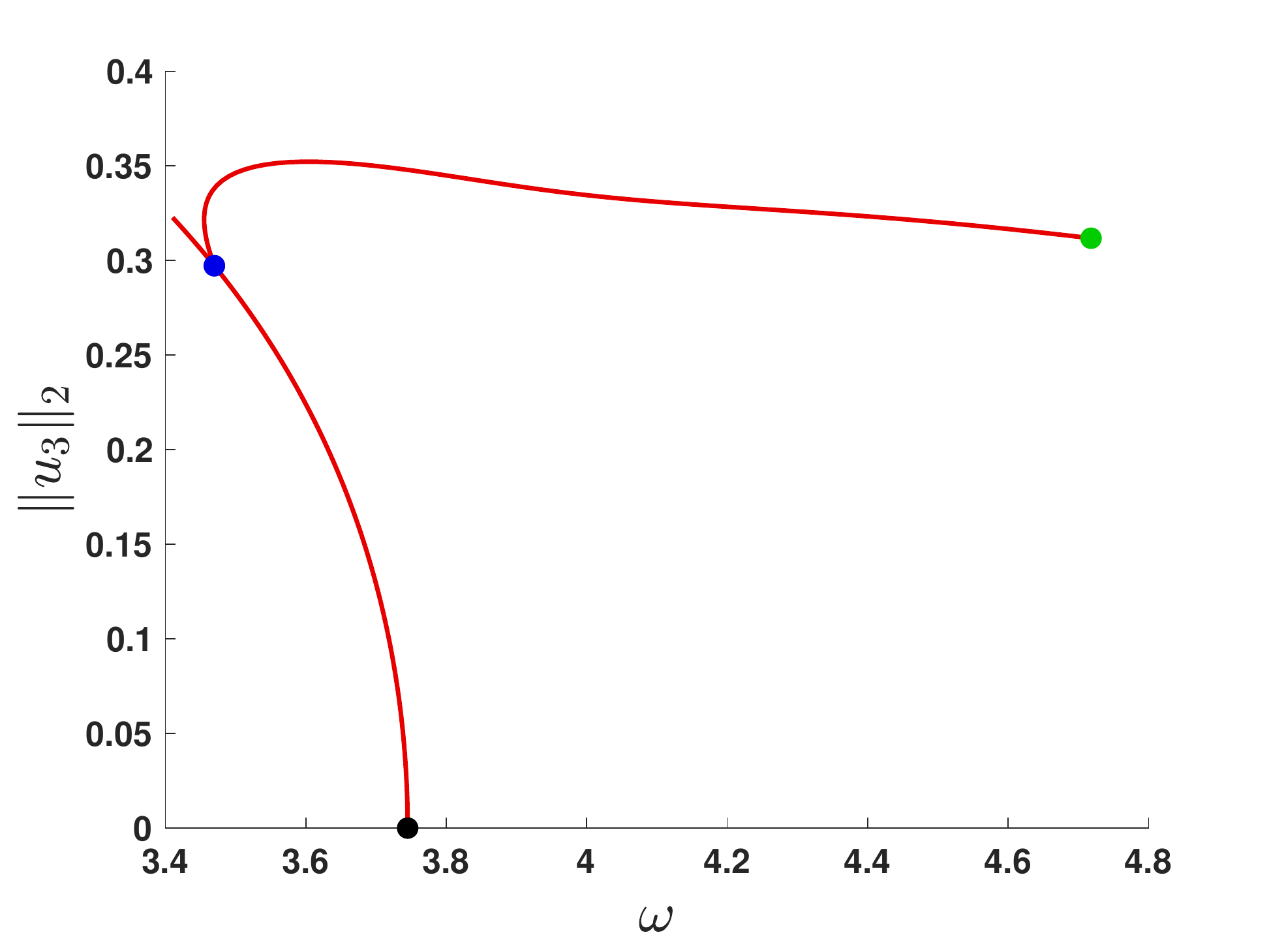}
\includegraphics[width=.32 \textwidth]{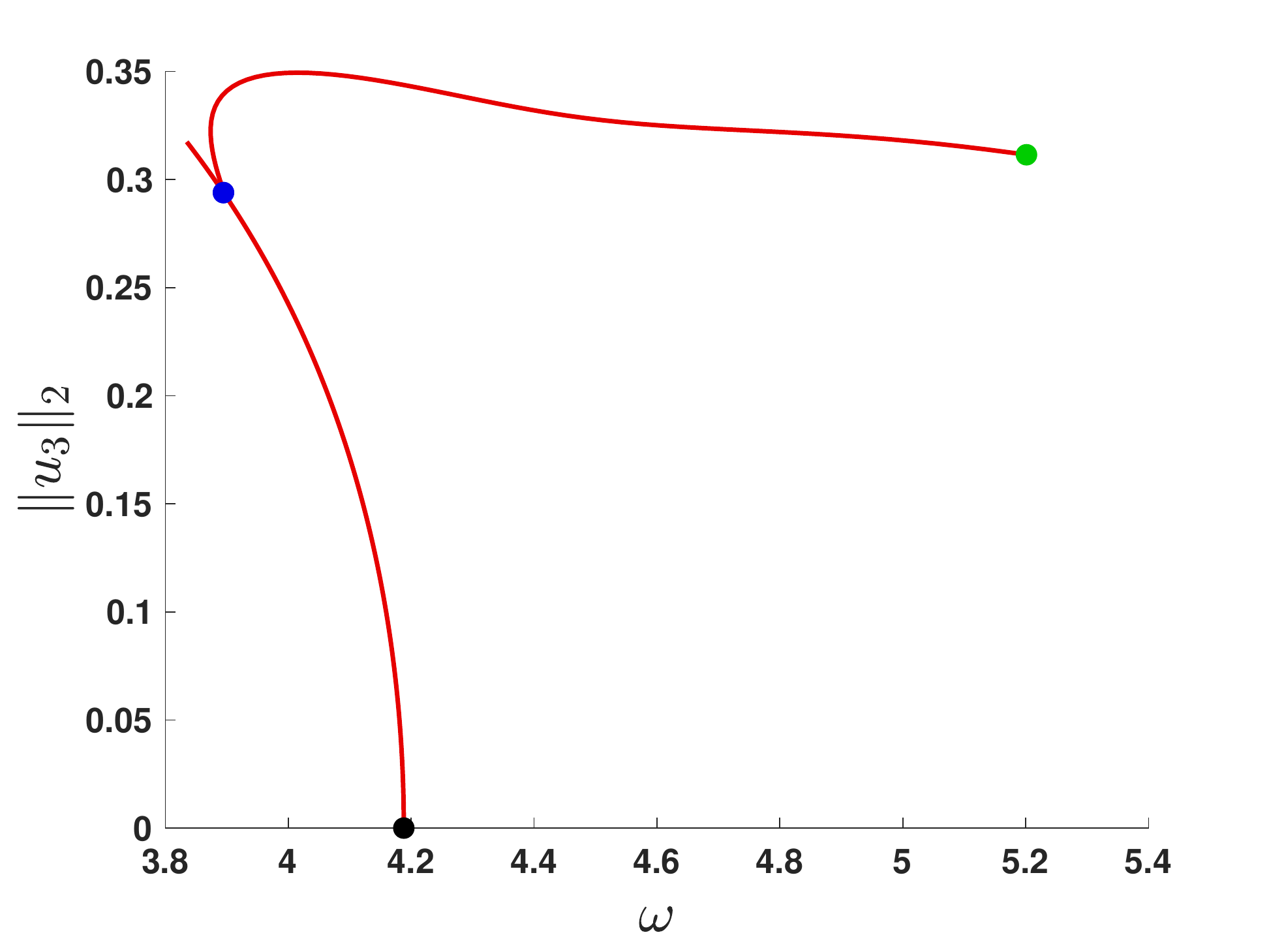}
\vspace{-.3cm}
\caption{Continuation from the triangle (the black dot) to the figure eight (the green dot) for $N=11$ (left), $N=13$ (center) and $N=15$ (right) bodies.}
\label{fig:11_13_15_bodies}
\end{figure}

We remark that bifurcation diagrams in the cases $N = 3, 5, 7, 9$ exhibit some 
qualitative differences.  Most notably the angle between the Lyapunov family and the 
axial family changes dramatically in these cases.   The diagrams in the cases
of $N = 9,11,13,15$ on the other hand are very similar, and exhibit a kind of 
convergence to a universal profile.  Again, the existence of such a profile is pure
conjecture at this point.  But the numerics seem to bear it out.

It was observed early on that the three body figure eight choreography is linearly stable (Morse index zero).
Indeed, the KAM stability of the three body eight was established
by Kapela and Sim\'{o} \cite{KaSi07,KaSi17} using computer-assisted methods.  
For five or more bodies no stable eights have ever been reported, and indeed we observe 
quite large Morse indices along the axial continuation branch in all but the three body case.


\begin{figure}[t!]
\centering
\vspace{\baselineskip}
\includegraphics[width=0.9 \textwidth]{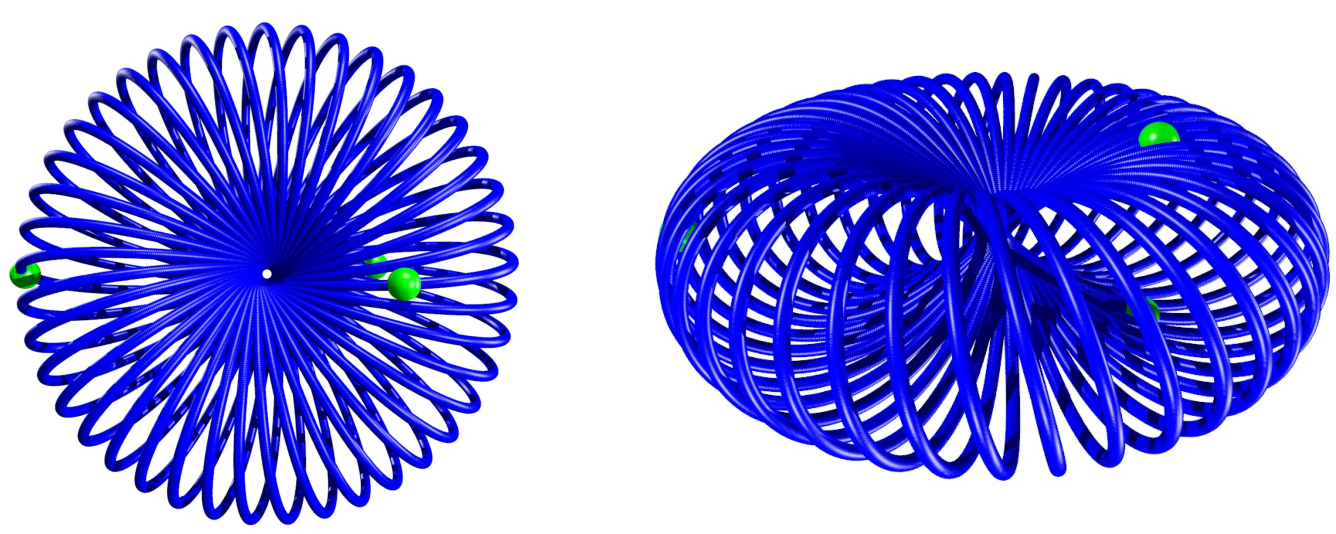}
\vspace{-.3cm}
\caption{\textbf{A stable spatial three body choreography:} a spatial choreography 
in the three body axial family near the three body eight.  The orbit is linearly stable 
in the sense of Hamiltonian systems: that is, all of it's Floquet multipliers are on 
the unit circle. This choreography is a $(p,q)$-torus knot with $p=19$ and $q= 41$.
See \cite{renatoAnimations} for an interactive animation.
}
\label{fig:stableSpatial}
\end{figure}

\begin{figure}[t!]
\centering
\vspace{\baselineskip}
\includegraphics[width=0.9 \textwidth]{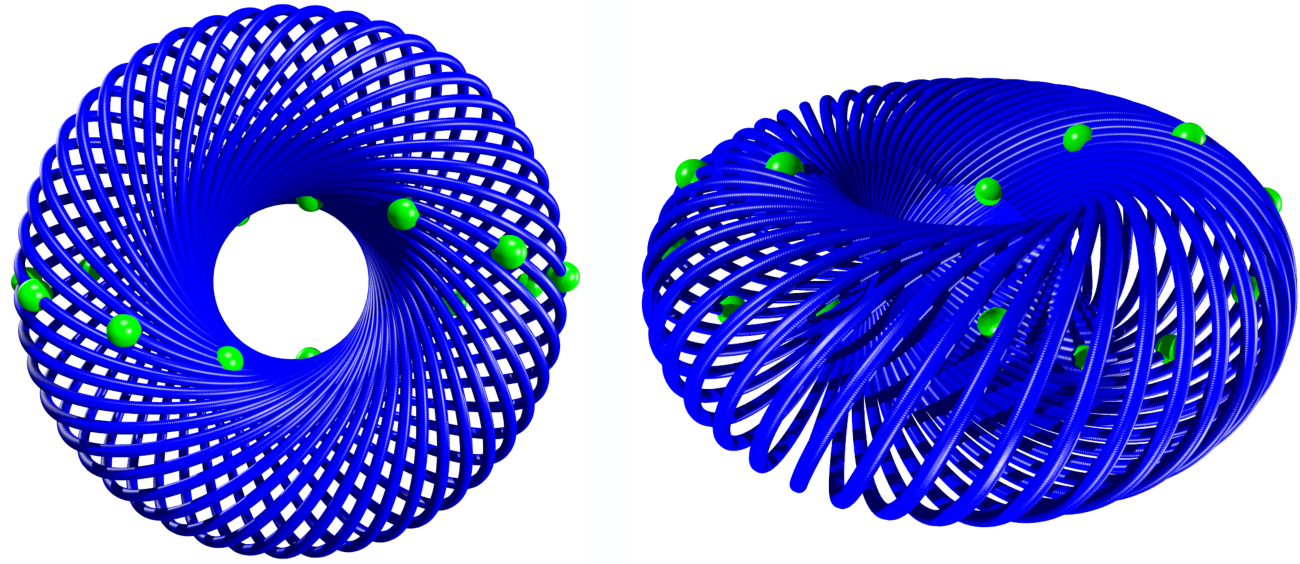}
\vspace{-.3cm}
\caption{\textbf{A spatial fifteen body choreography:} a spatial choreography 
in the fifteen body axial family along the branch that contains the fifteen body eight. The orbit is
a $(p,q)$-torus knot choreography with $p = 31$ and $q = 47$.}
\label{fig:stableSpatial}
\end{figure}

\begin{remark}[A stable spatial three body choreography] \label{rem:stableSpatialThreeBody}
A periodic orbit is linearly stable \textit{in the sense of Hamiltonian systems} if all of 
its Floquet multipliers are on the unit circle in $\mathbb{C}$.
It was observed early on, based on numerical evidence, 
that the three body figure eight seemed to be 
linearly stable in the sense of Hamiltonian systems.  This observation was 
eventually proven by Kapela and Sim\'{o} \cite{KaSi07} using computer
assisted methods of proof.  Indeed, the same authors prove KAM stability in 
\cite{KaSi17}, again using computer-assisted methods.  

It is an open question as to wether or not there exist other stable eights for higher numbers 
of bodies, and our numerical experiments seem to suggest that the answer to this 
question is ``no''.  On the other hand, linear stability is a robust property, so that 
nearby periodic orbits in the continuation class of the three body eight are also stable.
By numerically exploring the continuation class we have been able to find many spatial 
choreographies which are linearly stable.  One such orbit is illustrated in 
Figure \ref{fig:stableSpatial}.
\end{remark}

\section{Conclusion} \label{sec:conclusions}
Building on previous numerical studies of $n$-body eights
\cite{Mo93,Si00,MR1919833} and in particular 
numerical continuation results for choreographies found in 
\cite{MR2429679,CaDoGa18},
we applied classical numerical continuation methods to the periodic 
solutions of a delay differential equation (DDE) describing choreographic 
motion in the gravitational $n$-body problem.  
The DDE formulation is given \cite{ourTorusKnots}, and 
has two distinct advantages over working directly 
with the standard $n$-body equations of motion derived from Newton's Laws.  
In the first place, periodic solutions of the DDE 
satisfying a certain number theoretic condition correspond
to choreographies
rather than to arbitrary $n$-body periodic motions.  
Second, the DDE reduces to a system of six scalar equations (with delays), 
regardless of the number of bodies under consideration: 
adding more bodies introduces new terms to the 
nonlinearity rather than increasing the dimension of the system.  

The $n$-gon appears as a constant/equilibrium solution of the DDE, and a linear 
stability analysis shows that there is a single vertical Lyapunov 
family of periodic orbits in the center manifold.  
Using the explicit first order formulas for the vertical family derived in 
\cite{GaIz11,MR3554377} we are able to start the numerical continuation 
of the vertical family in an automatic way for any desired number of bodies.
For every odd $n$ between $3$ and $15$ we continue the vertical family 
with respect to energy/frequency, checking for the eight 
as we move along the branch. Note that we recover the earlier three and seven body results 
from \cite{MR2429679,CaDoGa18}, and also obtain new examples connecting 
the $n$-gon to the eight.

In each case we find that the $n$-body eight appears shortly after a symmetry 
breaking bifurcation from the $n$-gon's branch.   Moreover, after leaving
the $n$-gon using the formulas developed in \cite{GaIz11,MR3554377}, there 
appears to be one and only one bifurcation between the $n$-gon and the eight.
We stress that this is another 
advantage of performing the continuation in the DDE rather than using the 
$n$-body equations of motion:
continuation of choreographies in the full $n$-body problem can
result in additional bifurcations, as non-choreography periodic 
solutions may bifurcate from choreographies.  The fact that, in the 
symmetrized setting, a single dynamical 
mechanism appears to organize the transition from $n$-gon to eight 
leads us to generalize Marchal's conjecture (described in \cite{fijozHabilatation}
and again in the introduction of the present work)
to any odd number of bodies.  
  
We remark that -- even in the case of three bodies -- there appears
as of yet to be no 
mathematically rigorous proof of Marchal's conjecture, much less any of 
it's generalizations to more bodies.  
A very interesting avenue of future research would be to develop such proofs.
Mathematically rigorous results about $n$-body choreographies come in two 
main varieties: variational methods and computer-assisted proofs.
We refer the interested reader
to the works of \cite{ChMo00,BaTe04,TeVe07,FeTe04,BT04} for a
much more complete discussion of the literature on variational methods
for choreography problems.  Computer-assisted methods of 
proof for $n$-body choreographies which apply constructive
geometric arguments  in the full $n$-body state space are found in 
the works of \cite{KaZg03,MR2185163,KaSi07,KaSi17}.  A very interesting work 
which uses both variational and computer-assisted approaches is \cite{MR2259202}.  
The author's aforementioned 
work in \cite{ourTorusKnots} uses a-posteriori analysis for Fourier spectral methods 
to prove the existence of spatial torus knot choreographies using the delay differential 
equation set up exploited in the present work.

As is well known and already mentioned in the introduction,
 there are technical difficulties applying variational 
methods to continuation branches which encounter bifurcations.  On the other hand, 
mathematically rigorous computer-assisted methods for studying continuous branches 
of periodic orbits are by now quite advanced.  This is true even in the case of 
DDEs, and we refer for example to the work of  
\cite{MR2487806,MR2711226,jonopJP_Kon,MR2630003,Le18,JB_Elena_Cont_PO}.
Moreover,  methods of computer-assisted proof have been developed for proving the 
existence of, and continuing through, a number of infinite dimensional bifurcations.
See for example the work of
\cite{MR3779642,jono_jones,MR2679365,MR3792794,MR3808252}.

While existing methods from the works just cited do not cover the symmetry breaking 
bifurcations needed to prove instances of the generalized Marchal's conjecture,
we believe an appropriate framework for the envisioned proofs can be obtained by 
extending/adapting these works.  
Moreover, the out of plane bifurcation from the Lagrangian 
$n$-gon can be handled using the results of \cite{GaIz11,GaIz13}, as in the present work.
Computer-assisted proofs of several cases of the generalized Marchal's conjecture,
the ones for which we have presented numerical evidence in the present manuscript, 
are the subject of a work in preparation by the authors.

\end{document}